\documentclass[a4paper,11pt]{article}
\usepackage[a4paper]{geometry}
\usepackage{amsmath}
\usepackage{mathtools}
\usepackage{hyperref}
\usepackage{amsfonts}
\usepackage{amssymb,amsbsy,amsfonts,amscd}
\usepackage{amsthm}
\usepackage[T1]{fontenc}
\usepackage[english]{babel}
\usepackage{pst-node}
\usepackage{pst-coil}
\usepackage{pstricks}
\usepackage{graphicx}
\usepackage{pstricks-add}
\usepackage{url}
\usepackage{amsmath}
\usepackage{dsfont}
\usepackage{enumitem}

\numberwithin{equation}{section}
\newtheorem{definition}{Definition}[section]
\newtheorem{assumption}[definition]{Assumption}
\newtheorem{theorem}[definition]{Theorem}
\newtheorem{proposition}[definition]{Proposition}
\newtheorem{lemma}[definition]{Lemma}

\newcommand{\CX}{\mathcal{X}}
\newcommand{\CU}{\mathcal{U}}
\newcommand{\CE}{\mathcal{E}}
\newcommand{\dCX}{{\mathbb{X}}}
\newcommand{\R}{\mathbb{R}}
\newcommand{\N}{\mathbb{N}}

\newcommand{\Pb}{\mathbb{P}}
\newcommand{\D}{\mathbb{D}}
\newcommand{\Pbx}{\mathbb{P}_{\delta_x}}
\newcommand{\E}{\mathbb{E}}

\newcommand{\p}{\partial}
\newcommand{\de}{\delta}
\newcommand{\one}{\mathds{1}}

\newcommand{\CC}{\mathcal{C}}
\newcommand{\W}{{\mathcal{W}_1}}
\newcommand{\LL}{\mathcal{L}}

\author{H\'el\`ene Leman\thanks{CMAP, Ecole Polytechnique, UMR 7641, route de
    Saclay, 91128 Palaiseau Cedex-France; E-mail: \texttt{helene.leman@polytechnique.edu}}}
\title{Convergence of an infinite dimensional stochastic process to a spatially structured trait substitution sequence}
\date\today

\begin{document}

\maketitle

\begin{abstract}
 We consider an individual-based spatially structured population for Darwinian evolution in an asexual population. The individuals move randomly on a bounded continuous space according to a reflected brownian motion. The dynamics involves also a birth rate, a density-dependent logistic death rate and a probability of mutation at each birth event. We study the convergence of the microscopic process when the population size grows to $+\infty$ and the mutation probability decreases to $0$. We prove a convergence towards a jump process that jumps in the infinite dimensional space of the stable spatial distributions. The proof requires specific studies of the microscopic model. First, we examine the large deviation principle around the deterministic large population limit of the microscopic process. Then, we find a lower bound on the exit time of a neighborhood of a stationary spatial distribution. Finally, we study the extinction time of the branching diffusion processes that approximate small size populations.
\end{abstract}

\normalsize

\section{Introduction}
\label{sec_intro}

The spatial aspect is an important issue in ecology \cite{tilman_kareiva_1997,durrett_levin_1998}. The influence of the heterogeneity of the environment on the phenotypic evolution has been explored for a long time \cite{endler_1977,futuyma_moreno_1988,kassen_2002}. For example, the emergence of phenotypic clusters under a heterogeneous space has been extensively studied \cite{doebeli_dieckmann_2003, polechova_barton_2005, leimar_doebeli_dieckmann_2008}. In \cite{doebeli_dieckmann_2003, leimar_doebeli_dieckmann_2008}, the authors suggest that clustering and aggregation of individuals can be a consequence of the spatial competition between individuals. Those phenomena generate a structured population based on isolated patches. In \cite{polechova_barton_2005}, the authors draw attention to the influence of the boundary of the spatial environment. The sensibility to heterogeneously distributed resources is also a key point to study the spatial dynamics of population \cite{grant_grant_2002}. In this context, the effect of a spatial structure on the evolution of a population is fundamental.\\
\indent In this paper, we use a population model that describes the interplay between evolution and spatial structure. We are interested in a macroscopic approximation of the microscopic model under three main biological assumptions~: rare mutations, large population size and the impossibility of coexistence of two traits for a long time scale. The main result of this paper implies a convergence of the microscopic model to a jump process that jumps in an infinite dimensional space characterized by the spatial profiles of the population. This result is correlated with several works on adaptive dynamics and in particular with the model of Trait Substitution Sequence (TSS). Metz and al. \cite{metz_geritz_MJV_1996} have introduced this model from an ecological point of view. It describes the evolution of phenotypic traits in the case where the ecological time scale, related to the dynamics of the population, and the evolutionary time scale, related to the mutations, are separated. 
In the evolution time scale, the model describes the succession of invading phenotypic traits as a jump Markov process in the space of phenotypic traits. The link between the microscopic model and the TSS model has been completely proved by Champagnat \cite{champagnat_2006} in a simpler ecological context. Recently, some papers have generalized this approach in the case of an aged-structured population \cite{tran_2008}, of a multi-resources chemostat model \cite{champagnat_jabin_meleard_2014} or of a prey-predator model \cite{costa_hauzy_2014}. But only \cite{tran_2008} deals with some processes with values in infinite dimensional spaces, and the age structure is deterministic. In contrast with it, we are concerned with the spatial aspect of a population living on a heterogeneous environment where the individuals move randomly.\\
\indent We study an individual-based model in which any individual birth and death events are described. This Markov process has been initially introduced by Champagnat and Méléard in \cite{champagnat_meleard_2007}. For any time $t\geq 0$, each individual $i$ is described by two characteristics $(X^i_t, U^i_t)$. $X^i_t$ represents its location in an open, bounded and convex subset $\CX$ of $\R^d$ with a $C^2$-boundary. $U^i_t$ denotes its phenotypic trait which belongs to a compact subset $\CU$ of $\R^q$.  The phenotype of an individual does not change during its life time contrary to its location. The location space may represent a geographic landscape or a theoretical space that describes a gradient of temperature, a gradient of elevation or a resource parameter as seed size for a population of birds \cite{doebeli_dieckmann_2003}. In the context of the last example, the phenotype may represent the beak size of a bird and it is interesting to study the adequacy between the beak size and the seed size when evolution occurs \cite{grant_grant_2002}. \\
\indent The total population is represented at any time $t$ by the finite measure
\begin{equation}
 \nu^K_t =\frac{1}{K}\sum \limits_{i=1}^{N_t} \delta_{(X_t^i,U_t^i)},
\end{equation}
where $\delta_y$ corresponds to the Dirac measure at $y$, $N_t$ is the number of individuals at time $t$. The parameter $K$ scales the population size and the biological assumption of large population size is stated into mathematics by $K$ tends to $+\infty$.\\
\indent The dynamics of the process is driven by a birth and death diffusion process, in which the motion, birth, mutation and death of each individual depends on its location and trait.\\
\indent 
Any individual $i$ with phenotypic trait $u$ moves according to a diffusion process driven by the following stochastic differential equation normally reflected at the boundary $\p\CX$,
\begin{equation}
\label{eq_brownien}
 %\left\{ \begin{array}{ccl}
         dX^i_t=\sqrt{2m^u} Id \cdot dB_t  - n(X^i_t)dl_t \\
         %\vert r \vert_t &=&\int \limits_0^t 1_{\{ X_s \in \partial \CX\}} d\vert r \vert_s ; \; r_t=\int \limits_0^t n(X_s)d\vert r\vert_s.
        %\end{array} \right.
\end{equation}
where $B$ is a $d$-dimensional brownian motion, $l_t$ is an adapted continuous and non-decreasing process with $l_0 = 0$, it increases only when $X^i_t \in \p \CX$ and the diffusion coefficient $m^u$ is a function of the trait.\\
\indent We consider a population with asexual reproduction. An individual with location $x\in \CX$ and trait $u\in \CU$ gives birth at rate $b(x,u)$. This rate can also be denoted by $b^u(x)$ when $u$ is fixed and it is assumed to be bounded.
The offspring appears at the location of its parents. Furthermore, a mutation may occur with probability $q_Kp$, making the phenotypic trait of the offspring different.
%With probability $(1-q_K \cdot p)$, the offspring inherits the parameters of its parent $(x,u)$ at its birth time.
%With probability $q_K \cdot p$, the offspring is a mutant individual with parameters $(x,w)$ at its birth time. 
The law of the mutant trait is then given by a kernel $k(x,u,\cdot)$.
The mutation probability $p$ may depend on the trait and the location.
The parameter $q_K$ scales the mutation probability and the biological assumption of rare mutations is stated by $q_K \to 0$.\\
\indent The death rate depends on the characteristics of the individual and on the competition between all individuals. The natural death rate is $d(x,u)$. The competition exerted by an individual $(y,v)$ on an individual $(x,u)$ depends on the location $y$ and on the two traits through a competition kernel $c:\CU \times \CX \times \CU  \to \R^+$. For the population $\nu=\frac{1}{K}\sum_{i=1}^n \delta_{(x_i,u_i)} \in M_F(\CX\times\CU)$, the competitive pressure exerted on individual $(x,u)$ is
\begin{eqnarray*}
 c\cdot\nu(x,u) =\dfrac{1}{K} \sum_{i=1}^n c(u,x_i,u_i) =  \dfrac{1}{K} \int \limits_{\CX \times \CU} c(u,y,v) \nu(dy,dv).
\end{eqnarray*}
Remark that the competition kernel does depend on $y$. This spatial dependence yields non-trivial mathematical difficulties. In Champagnat-Méléard~\cite{champagnat_meleard_2007}, the competition kernel depends also on $x$ but the long time behavior of the deterministic limit is still unknown, to our knowledge. 
Finally, the total death rate is $d(x,u)+c\cdot \nu(x,u)$. 
As for the birth rate, $d(x,u)$ can also be denoted by $d^u(x)$ and $c(u,y,v)$ by $c^{uv}(y)$ when $u$ and $v$ are fixed.\\
Let us state the assumptions on the parameters.
\begin{assumption}
\label{ass_coef}
\begin{enumerate}[noitemsep]
 \item $m$, $b$, $d$, $k$ and $c$ are continuous and non-negative on their domains and $b$, $d$ and $c$ are Lipschitz functions with respect to $x$ and $y$.
 \item There exist $\bar m$, $\underbar b$, $ \bar b$, $\bar d$, $\bar c$, $\underbar c$, $\bar k \in \R$ such that for any $ (x,u,y,v) \in (\bar\CX \times \CU)^2$, 
$ 0 < m^u \leq \bar m,\quad \underbar b < b(x,u) \leq \bar b, \quad d(x,u) \leq \bar d,\quad \underbar c \leq c(u,y,v) \leq  \bar c, \quad k(x,u,v) \leq \bar k ,$
and $d$ is not the zero function.
\item The sequence of initial measures $(\nu_0^K)_{K>0}$, which belongs to $M_F(\CX\times\CU)$, converges in law to some deterministic measure denoted by $\xi_0$ and it satisfies $\sup_K \E[\langle \nu_0^K,1 \rangle^3] < +\infty$.
\item $q_K$ tends to $0$ when $K$ tends to $+\infty$.
\end{enumerate}
\end{assumption}

\indent Before going further, let us set and recall the notation, which we use in the entire paper.\\
\noindent
\textbf{Notation}
\begin{itemize}[noitemsep]
 \item For all $x \in \partial\CX$, $n(x)$ denotes the outward normal to the boundary of $ \CX$ at point $x$.
 \item For sufficiently smooth $f$ and for all $(x,u)\in \partial\CX\times \CU$, $\partial_nf(x,u)$ denotes the scalar product $\nabla_xf(x,u) \cdot n(x)$.
 \item $C^{k,l}_n(\CX \times \CU)$ represents the set of functions $f$ such that $f\in C^{k,l}(\CX \times \CU)$ and $\partial_nf(x,u)=0$ for all $(x,u) \in \partial\bar{\CX}\times \CU$. We define $C^{k,l,j}_n(\CX \times \CU \times [0,T])$ similarly.% as the set of functions $f$ such that $f\in C^{k,l,j}(\CX \times \CU\times [0,T])$ and $\partial_nf(x,u,t)=0$ for all $(x,u,t) \in \partial\bar{\CX}\times \CU\times[0,T]$.
 \item For any $f\in C^{k,l,j}_n(\CX \times \CU \times [0,T])$, $f_s$ is the function on $\CX \times \CU$ such that $f_s(x,u)=f(x,u,s)$.
 \item For any compact set $\dCX$, we denote the space of finite measures on $\dCX$ by $M_F(\dCX)$.
 \item $C^{Lip}(\dCX)$ denotes the set of all positive Lipschitz-continuous functions $f$ on $\dCX$ bounded by $1$ and with a Lipschitz constant smaller than $1$.
 \item We define the Kantorovich-Rubinstein distance on $M_F(\dCX)$ by~: for any $\nu, \mu \in M_F(\dCX)$,
 \begin{equation*} 
 \W(\nu,\mu)=\sup_{f\in C^{Lip}(\dCX)} \left\vert \int_{\dCX}fd\nu-\int_{\dCX}fd\mu \right\vert.
 \end{equation*}
 As $\dCX$ is a compact set, this metric is a metrization of the topology of weak convergence. It is equivalent to the $1^{st}$-Wasserstein distance.
 \item $B(\nu ,\gamma)$ represents the ball of center $\nu$ and radius $\gamma$ in $M_F(\dCX)$ for the previous distance.
  \item $\mathbb{D}([0,T],M_F(\dCX))$ denote the space of càdlàg functions from $[0,T]$ to $M_F(\dCX)$, equipped with the Skorokhod topology. 
 \item For any $\xi\in M_F(\CX\times\{u,v\})$, we identify the two following ways of writing~: $\xi(dx,dw)=\xi^u(dx)\delta_u(dw)+\xi^v(dx)\delta_v(dw)$ and $\xi=(\xi^u,\xi^v)\in (M_F(\CX))^2$.
\end{itemize}

\section{Main theorem}
\label{sec_theo}

\indent A full algorithmic description and a mathematical formulation of the model described in the previous part are detailed in Champagnat and Méléard \cite{champagnat_meleard_2007}. Moreover, a macroscopic approximation has been proved as a large population limit.
\begin{theorem}[Theorems 4.2 and 4.6 in \cite{champagnat_meleard_2007}]
\label{theo_limitpop}
Suppose that Assumption~\ref{ass_coef} holds. For all $T>0$, the sequence $(\nu^K)_{K>0}$ of processes belonging to $\mathbb{D}([0,T],M_F(\CX\times \CU))$ converges in law to a deterministic and continuous function $\xi$, i.e. $\xi \in \mathbb{C}([0,T],M_F(\CX\times \CU))$ such that $\sup_{t \in [0,T]} \langle \xi_t,1 \rangle < +\infty$ and  $\forall f \in C^{2,0}_n(\CX\times \CU)$,
\begin{multline}
 \langle \xi_t,f \rangle =\langle \xi_0,f \rangle +\int \limits_0^t \int \limits_{\CX\times \CU} \bigg\{ m^u \Delta_xf(x,u)+\big[ b(x,u)-d(x,u)-c\cdot\xi_s(x,u) \big] f(x,u) \bigg\} \xi_s(dx,du)ds.
\label{eq_theoapproxpop}
\end{multline}
Moreover, if $\CU$ is finite, for any $u\in\CU$ and $t>0$, $\xi_t(.,u)$ has a density with respect to Lebesgue measure which is a $C^2$-function.
\end{theorem}

\indent The limiting equation~\eqref{eq_theoapproxpop} is a nonlinear nonlocal reaction-diffusion equation defined on the space of traits and locations. In \cite{desvillettes_ferriere_prevost_2004, arnold_desvillettes_prevost_2012}, the authors have studied the existence of the steady states of similar equations in the context of frequent mutations. Our study involves a rare mutations assumption and mutation terms disappear in the limit. The stability of the steady states and the long time behavior of the solutions to \eqref{eq_theoapproxpop} have been characterized in \cite{leman_mirrahimi_meleard_2014, coville_2013} in the particular cases of a monomorphic population (all individuals have a same phenotype) 
 and a dimorphic population (two traits are involved). The stationary states and their stability are described using the following parameters.
\begin{definition}
\label{def_equilibrium}
For any $u\in \CU$, we define $H^u$ by,
\begin{equation}
\label{eq_defH}
 H^u=-\min_{\phi \in H^1(\CX), \phi \not\equiv 0} \dfrac{1}{\|\phi \|^2_{L^2(\CX)}} \left[ m^u\int_{\mathcal{X}} |\nabla \phi(x) |^2dx-\int_{\mathcal{X}}(b^u-d^u)(x)\phi^2(x)dx\right],
\end{equation}
where $H^1(\CX)$ is the Sobolev space of order $1$ on $\CX$. $H^u$ is thus the principal eigenvalue of the operator $m^u\Delta_x \cdot +(b^u-d^u) \cdot $ with Neumann boundary condition on $\CX$.
Let $\bar g^u \in H^1(\CX) $ be the eigenfunction of the previous operator associated with the eigenvalue $H^u$ such that 
$$\int_{\CX} c^{uu}(y)\bar g^u(y)dy=H^u.$$
According to~\cite{leman_mirrahimi_meleard_2014}, $\bar g^u \in C^1(\CX)$. If $\bar g^u\geq 0$, we define the associated measure in $M_F(\CX)$
$$
\bar\xi^u(dx):=\bar g^u(x)dx.
$$
Finally, for any $(u,v)\in \CU$, we set
\begin{equation*}
 \kappa^{vu}:={\int_{\CX} c^{vu}(y)\bar g^u(y)dy}\left({\int_{\CX} \bar g^u(y)dy}\right)^{-1}.
\end{equation*}
\end{definition}
As proved in \cite{leman_mirrahimi_meleard_2014}, $H^u>0$ is the condition ensuring that a monomorphic population with trait $u$ is able to survive for a long time at the ecological time scale. In that case, the stationary stable state is described by the positive spatial profile $\bar g^u$. \\
The dimorphic case implies four distinct stationary states~: the trivial state $(0,0)$, two monomorphic states and one co-existence state. To ensure that the co-existence state is unstable, we set the following assumption. 
\begin{assumption}
\label{ass_noncoexistence}
 Let $u$ be in $\CU$, for almost all $v\in \CU$,
\begin{enumerate}
 \item \label{cond_1} either, $H^v \kappa^{uu} -H^u \kappa^{vu} < 0$, 
 \item \label{cond_2} or, $\left\{ \begin{aligned} & H^v \kappa^{uu} -H^u \kappa^{vu}>0 \\ & H^u \kappa^{vv} -H^v \kappa^{uv}<0.  \end{aligned}\right.$
\end{enumerate}
\end{assumption}
Using \cite{leman_mirrahimi_meleard_2014}, we notice that Assumption~\ref{ass_noncoexistence} states into mathematics the impossibility of co-existence of two traits for a long time, this assumption is also known as the "Invasion-Implies-Fixation" principle. Under Assumption~\ref{ass_noncoexistence}, any solution to \eqref{eq_theoapproxpop} converges either to $(\bar g^u,0)$, or to $(0,\bar g^v)$, which are two monomorphic states.
More precisely, Condition~\ref{cond_1} ensures the stability of the equilibrium $(\bar g^u, 0)$. Thus, if a mutant population with phenotype $v$ is emerging in a monomorphic well-established population with phenotype $u$, it will not be able to survive. 
Under Condition~\ref{cond_2}, the authors of \cite{leman_mirrahimi_meleard_2014} prove that the deterministic solution to \eqref{eq_theoapproxpop} converges to the stable equilibrium $(0,\bar g^v)$ whatever the initial condition is.
In the light of the previous considerations, we refer to $H^v \kappa^{uu} -H^u \kappa^{vu}$ as the invasion fitness of the individuals with type $v$ in a resident population with type $u$.
Furthermore, the probability of success of such an invasion is described precisely 
by means of the geographical birth position $x_0$ of the first individual with trait $v$ and the function $\phi^{vu}$ defined below. That probability is precisely $\phi^{vu}(x_0)$. 
\begin{definition}
\label{def_phivu}
For any $u, v\in \CU$, $\phi^{vu}$ is the function on $\CX$ such that
\begin{enumerate}
 \item If $H^v \kappa^{uu} -H^u \kappa^{vu} \leq 0$, $\phi^{vu}(x)=0$ for all $x\in \CX$.
 \item If $H^v \kappa^{uu} -H^u \kappa^{vu} > 0$, $\phi^{vu}$ is the unique positive solution to the elliptic equation
\begin{equation}
\label{eq_survivalprob}
 \left\{ 
\begin{aligned}
 &m^v \Delta_x \phi(x) + \left(b^v(x)-d^v(x)-\int_{\CX}c^{vu}(y)\bar g^u(y)dy\right)\phi(x)-b^v(x) \phi(x)^2=0, \forall x \in \CX,\\
&\p_n \phi (x)=0, \forall x \in \p \CX.
\end{aligned}
\right.
\end{equation}
\end{enumerate}
\end{definition}

We are now ready to state the main result of this paper. 
\begin{theorem}
\label{theo_main}
We suppose that Assumptions~\ref{ass_coef} and~\ref{ass_noncoexistence} hold. We also assume that the scaling parameters satisfy
\begin{equation}
\label{ass_KqK}
K q_K\log(K) \underset{K\to +\infty}{\longrightarrow} +\infty \; \text{ and } \; K q_K e^{KV} \underset{K\to +\infty}{\longrightarrow} 0, \text{ for any } V>0.
\end{equation}
Then for any $T>0$, $\left(\nu^K_{({t}/{Kq_K})}\right)_{t\in[0,T]}$ converges towards a jump Markov process $(\Lambda_t)_{t\geq [0,T]}$ as $K\to+\infty$. At any time $t$, $\Lambda_t$ belongs to the subspace $\{ \bar\xi^u\delta_u, u\in \CU\}$ of $ M_F(\CX\times \CU)$, where for any $u\in \CU$, $\bar\xi^u\in M_F(\CX)$ is the spatial pattern defined in Definition~\ref{def_equilibrium}. The process $(\Lambda_t)_{t\geq 0}$ jumps from the state characterized by the trait $u\in \CU$ to the state characterized by $v\in \CU$ at the infinitesimal rate
\begin{equation*}
\int_{\CX} p b^u(x)\phi^{vu}(x)\bar g^u(x)k(x,u,v)dx dv.
\end{equation*}
This convergence holds in the sense of convergence of the finite dimensional distributions.
\end{theorem}
Remark that $\bar \xi^u$ describes the spatial distribution of the monomorphic population with trait $u$. The limiting process jumps from a spatial distribution to another one depending on the mutant trait. It models an evolutionary phenomenon using a sequence of monomorphic equilibria described by their spatial patterns.\\
\indent Although the structure of Theorem~\ref{theo_main}'s proof is similar to the one of Theorem 1 in \cite{champagnat_2006}, the spatial structure of the process leads us to deal with infinite dimensional processes. Two key points of the proof have to be approached differently. The first point concerns the study of the process when it is close to a monomorphic deterministic equilibrium. The aim is to estimate the exit time of a neighborhood of a stationary state to \eqref{eq_theoapproxpop}. We give the behavior of the stochastic process around its deterministic equilibrium taking into account a small mutant population and the possibility of other mutations. Thanks to it, we avoid the comparisons used in \cite{champagnat_2006}, where the behavior of the resident population process is compared with the behavior of a theoretical monomorphic population evolving alone. Those comparisons are much more involved when the population is spatially structured. Moreover, to estimate this exit time, we have to study a large deviation principle of the stochastic process $(\nu^K_t)_{t\geq 0}$ around its deterministic limit~\eqref{eq_theoapproxpop} when $K$ is large. The large deviations studies for processes combining diffusion process and jumps are still unresolved, to our knowledge. Those studies have thus their own interest. 
The second point which is approached differently concerns the study of small population size processes. The aim is to understand the dynamics of a population descended from a mutant which has appeared in a well-established monomorphic population. As long as the mutant population size is small, the competitive terms between mutants can be neglected. Thus, the dynamics of the mutant population can be compared with the dynamics of a branching diffusion process. We describe finely the survival probability of a branching diffusion process by means of the eigenparameters defined previously and we link it with the conditions presented in Assumption~\ref{ass_noncoexistence}.\\
\indent In Section~\ref{sec:largedeviation}, we explicit the upper bound of the large deviation principle by using ideas in \cite{dembo_zeitouni_1998, leonard_1995, tran_2008}. Then we study the functional rate associated with the large deviation principle. 
Section~\ref{sec_lowerbound} deals with the exit time of a neighborhood of a stationary state to \eqref{eq_theoapproxpop}. 
In Section~\ref{sec_survivalproba}, we study a branching diffusion process. First, we evaluate its probability of survival. Then, we characterize the scale time under which its size is of order $K$.
Section~\ref{sec_TSS} is devoted to the proof of Theorem~\ref{theo_main}. We detail two key propositions. The first one deals with the dynamics of the individual-based process in the case where only one trait are involved. The second one gives the dynamics of the process after the time of the first mutation but as long as only at most two traits are involved.
Finally, Section~\ref{sec_numerics} presents a numerical example that illustrates Theorem~\ref{theo_main}.\\

\section{Exponential deviations results}
\label{sec:largedeviation}

In this section, we are concerned by the large deviations from the large population limit \eqref{eq_theoapproxpop} for the process $(\nu^K_t)_{t\in [0,T]}$ when $K$ tends to $+\infty$ and $q_K$ tends to $0$. First of all, Theorem~\ref{theo_majoration} gives the upper bound of the large deviations principle. This theorem involves a rate function.
Let us first explicit it.
That requires specific notation which will be only used in this subsection~: let us fix $T>0$, 
\begin{itemize}[noitemsep]
\item $\CE=\CX\times \CU \times \{1,2\}$.
\item $\psi$ is the mapping such that for any function $f\in C^{2,0,1}(\CX\times \CU\times [0,T])$, for any $(x,u,\pi,t)\in \CE \times [0,T]$,
\begin{equation*}
 \psi(f)(x,u,t,\pi)= \left\{
 \begin{aligned}
  &f(x,u,t) &\text{ if } \pi=1,\\
  &-f(x,u,t) &\text{ if } \pi=2.
 \end{aligned}
\right.
\end{equation*}
\item For all $\nu=(\nu_t)_{t\in[0,T]} \in \mathbb{D}([0,T],M_F(\CX\times \CU))$, we define the positive finite measure 
\begin{align}
\label{eq_defmu}
 \mu^{\nu}_t(dx,du,d\pi)=\big[b(x,u)\delta_1(d\pi)
+(d(x,u)+c\cdot\nu_{t-}(x,u))\delta_2(d\pi)\big] \nu_{t-}(dx,du).
\end{align}
\item Finally, we introduce the log-Laplace transform $\rho$ of a centered Poisson distribution with parameter $1$,
$ \rho(x)=e^x-x-1,$
and its Legendre transform $\rho^*$,
\begin{equation*}
 \rho^*(y)=((y+1) \log(y+1)-y)\one_{\{y>-1\}}+\one_{\{y=-1\}} +\infty \cdot \one_{\{y<-1\}}. 
\end{equation*}
\end{itemize}

\bigskip

\noindent
We are now ready to define the rate function in which we shall be interested~:
for all $\xi_0 \in M_F(\CX\times\CU)$ and $\nu\in \mathbb{D}([0,T], M_F(\CX\times\CU))$, 
\begin{equation}
\label{def:Inu}
 I^T_{\xi_0}(\nu):=\left\{ 
 \begin{aligned}
  &\sup_{f \in C^{2,0,1}_n(\CX\times \CU \times [0,T])} I^{f,T}(\nu), \text{ if } \nu_0=\xi_0\\
  & +\infty, \text{ otherwise},
 \end{aligned}
\right.
\end{equation}
where
\begin{equation*}
\begin{aligned}
 I^{f,T}(\nu):=& \langle \nu_T, f_T \rangle -\langle \nu_0, f_0 \rangle -\int_0^T \langle  m \Delta_x f_s+m|\nabla_xf_s|^2 + \dfrac{\partial f_s}{\partial s} ,\nu_s\rangle ds\\
&-\int_0^T \int_{\CE} \Big(\psi(f)(x,u,s,\pi)+\rho(\psi(f)(x,u,s,\pi))\Big) d\mu^{\nu}_s ds.
 \end{aligned}
\end{equation*}
When there is no ambiguity, we will write $I^T(\nu)$ instead of $I^T_{\nu_0}(\nu)$.

\begin{theorem}
\label{theo_majoration}
 Suppose that Assumptions~\ref{ass_coef} holds.
 For all $\alpha >0$, $\xi_0 \in M_F(\CX\times\CU)$, for all compact set $\CC \subset B(\xi_0,\alpha)$, for all measurable subset $A$ of $\mathbb{D}([0,T], M_F(\CX\times\CU))$ such that there exists $M>0$ with $A\subset \{ \nu | \sup_{t\in [0,T]}\langle \nu_t,\one \rangle \leq M\} $,
\begin{equation}
\label{eq_GDavecsup}
 \underset{{K\to +\infty}}{\limsup} \; \frac{1}{K} \;\underset{\nu_0^K \in \CC \cap M_F^K}{\sup}\; \log \Pb_{\nu^K_0}(\nu^K \in A) \leq -\underset{\xi \in \CC, \nu \in \bar{A}}{\inf} \; I^T_{\xi}(\nu),
\end{equation}
where $M_F^K=\{\frac{1}{K} \sum_{i=1}^N\delta_{(x_i,u_i)}, \text{ with } N\in \N, (x_i,u_i)\in \CX \times \CU \}$.
\end{theorem}

\begin{proof}
 We will show the following upper bound 
 \begin{equation}
\label{eq_GD}
 %- \underset{\nu \in \mathring{B}}{\inf} \; I^T_{\xi_0}(\nu) \leq \underset{K \to +\infty}{\liminf} \; \frac{1}{K} \log \Pb(\nu^K \in B) \leq 
\underset{K \to +\infty}{\limsup} \; \frac{1}{K} \log \Pb(\nu^K \in A) \leq - \underset{\nu \in \bar{A}}{\inf} \; I^T_{\xi_0}(\nu),
 \end{equation}
indeed \eqref{eq_GDavecsup} can be directly deduced from this bound by a similar reasoning as in the proof of Corollary 5.6.15 in Dembo and Zeitouni \cite{dembo_zeitouni_1998}.
To prove \eqref{eq_GD}, we need the exponential tightness of the process $(\nu^K_t)_{t\in [0,T]}$ which is described by the following lemma.
 \begin{lemma}
\label{lemma_tensionexpo}
Suppose that Assumption~\ref{ass_coef} holds, and that there exists $C_{init}>0$ such that $\sup_{K\in \N} \langle \nu_0^K, \one \rangle < C_{init}$ a.s.. Then for all $L>0$, there exists a compact subset $\CC_L$ of the Skorohod space $\mathbb{D}([0,T],M_F(\CX\times\CU))$ such that
 \begin{equation*}
  \underset{K \to +\infty}{\limsup} \frac{1}{K} \log\Pb(\nu^K \not\in \CC_L) \leq -L
 \end{equation*}
\end{lemma} 
\noindent
We do not detail the proof of Lemma~\ref{lemma_tensionexpo} as it may be easily adapted from \cite{dawson_gartner_1987,graham_meleard_1997,tran_2008}. 
Then, set $\tau_M^K=\inf\{t\geq 0, \langle \nu^K_{t},\one \rangle \geq M\}$. Note that Lemma~\ref{lemma_tensionexpo} is also true for $(\nu^K_{t \wedge \tau^K_M})_{t\geq 0}$. Using a proof similar to Theorem 4.4.2 of \cite{dembo_zeitouni_1998}, we deduce the inequality
\begin{align*}
 \underset{K\to +\infty}{\limsup} \frac{1}{K} \log \Pb(\nu^K \in A)&=\underset{K\to +\infty}{\limsup} \frac{1}{K} \log \Pb(\nu^K_{.\wedge \tau^K_M} \in A)\\
 & \leq -\inf_{\nu\in \bar A} \Big( \sup_{f\in C^{2,0,1}_n(\CX\times\CU\times [0,T])}  (I^{f,T}(\nu)-H(I^{f,T}))\Big),
\end{align*}
where $H(I^{f,T})=\limsup_{K\to +\infty} \frac{1}{K}\log \E[\exp(KI^{f,T}(\nu^K_{.\wedge \tau^K_M})]$. Let us show that $H(I^{f,T})=0$. Let
\begin{multline*}
 \mathcal{N}_T=\exp\bigg(K I^{f,T}(\nu^K_{.\wedge \tau^K_M})\\
 -q_K K \int_0^{T\wedge \tau^K_M} \left\langle \nu^K_s, pb^u(x)\left(\int_\CU \phi(f_s)(x,w)k(x,u,w)dw -\phi(f_s)(x,u)\right)\right\rangle ds \bigg),
\end{multline*}
where $\phi(x)=x+\rho(x)=e^x-1$. 
Itô's Formula implies that $(\mathcal{N}_T,T\geq 0)$ is a local martingale. The definition of $\tau^K_M$ implies that $\mathcal{N}_T$ is bounded. So it is a martingale of mean $1$ and there exists a constant $C(\|f\|_{\infty},M)>0$ such that
$$
 \exp(-q_K K C(\|f\|_{\infty},M)) \leq \E\left[\exp\left(K I^{f,T}(\nu^K_{.\wedge \tau^K_M})\right) \right] \leq \exp(q_K K C(\|f\|_{\infty},M)).
$$
We conclude easily, since $q_K$ tends to $0$ when $K\to +\infty$.
\end{proof}

\noindent
The aim is now to write the rate function under a non-variational integral formulation which is more workable than that of \eqref{def:Inu}. Firstly, this integral formulation is convenient to use Chasles' Theorem. Secondly, it will be used to bound from above the distance between a solution to \eqref{eq_theoapproxpop} and any $\nu$, this upper bound is proved below in Proposition \ref{prop_lowerboundI}. Those two points are required to prove the results about the exit time in Section~\ref{sec_lowerbound}.\\
Before writing the non-variational formulation, let us define two functional spaces.
\begin{itemize}
 \item The Orlicz space associated with $\rho^*$ is $L_{\rho^*,T}$ the set of all bounded and measurable functions $h$ on $\CE\times [0,T]$ such that
\begin{equation}
\label{def:normrho}
 \|h\|_{\rho^*,T}:=\inf \left\{ \alpha>0, \int_{\CE\times [0,T]} \rho^*\left(  \dfrac{|h|}{\alpha} \right) d\mu^{\nu}_s ds\leq 1  \right\}<+\infty.
\end{equation}
The Orlicz space associated with $\rho$ is defined on the same way.
\item $\LL^{2}_T$ is the set of functions $h\in L^2(\CX\times \CU\times [0,T], \R^d)$ such that
\begin{equation}
\label{def:normL2}
 \|h\|_{\LL^2,T}:=\left( \int_0^T 2\langle \nu_s,m |h_s|^2  \rangle ds \right)^{1/2}<\infty.
\end{equation}

\end{itemize}

\begin{theorem}
 \label{theo_functionalrate}
 Suppose that Assumption~\ref{ass_coef} holds. Let $T>0$ and $\nu\in \D([0,T],M_F(\CX \times \CU))$, such that $I^T_{\nu_0}(\nu)<+\infty$, then there exist two measurable functions $(h^\nu_1,h^\nu_2)\in L_{\rho^*,T} \times \LL^2_T$ such that for all $f\in C^{2,0,1}_n(\CX \times \CU\times [0,T])$,
\begin{multline}
\label{eq_theoeqnu}
 \langle \nu_t, f_t \rangle =\langle \nu_0, f_0 \rangle +\int_0^T \int_{\CE} (1+h^\nu_1(x,u,s,\pi))\psi(f)(x,u,s,\pi) d\mu^{\nu}_s ds \\ +\int_0^T \int_{\CX\times \CU} \Big( m^u \Delta_x f_s(x,u)+ 2m^u h^\nu_2(x,u,s) \cdot\nabla_xf_s(x,u) 
  + \dfrac{\partial f_s}{\partial s}(x,u) \Big)\nu_s(dx,du),
\end{multline}
and the rate function can be written as follows
\begin{equation}
\label{eq_theorate}
 I^T_{\nu_0}(\nu)=\int_0^T \int_\CE \rho^*(h^\nu_1) d\mu^{\nu}_s ds + \int_0^T  m \langle \nu_s,|h^\nu_2|^2 \rangle ds<+\infty.
\end{equation}
\end{theorem}
The proof of Theorem~\ref{theo_functionalrate} uses convex analysis arguments which can be adapted from Leonard \cite{leonard_2001, leonard_2001_bis}. We do not detail its proof but we give the main ideas. For all $\nu\in \D([0,T],M_F(\CX \times \CU))$, $I^T(\nu)$ is equal to the Legendre transform $\Gamma^*$ of
\begin{equation}
 \Gamma : \left(\begin{matrix}
  (\psi,\nabla_x)(C^{2,0,1}(\CX\times \CU\times [0,T])) & \rightarrow & \R \\
  (g_1,g_2) & \mapsto &  \int_0^T \int_{\CE} \rho(g_1)d\mu^\nu_sds + \int_0^T \langle \nu_s, m |g_2 |^2 \rangle ds\\
\end{matrix}
\right) .
\end{equation}
at a well chosen point $l_\nu$. If $l_\nu$ belongs to the interior of the set $dom \Gamma^*$ of linear maps $l$ with $\Gamma^*(l)<+\infty$, we can exhibit $l_\nu$ by means of the derivative of $\bar\Gamma$, the Legendre biconjugate of $\Gamma$. Studying directly $\bar\Gamma$ is difficult. The key point is thus to work on the product space $L_{\rho,T}\times \LL^2_T$. In this way, we can study the Legendre biconjugate of an extension of $\Gamma$ on that space, in order to deal with the diffusive part and the jumps part separately. The diffusive part is treated using ideas of Dawson and Gartner \cite{dawson_gartner_1987} and Fontbona \cite{fontbona_2004} whereas the jumps part is treated using ideas of Leonard \cite{leonard_2001, leonard_2001_bis}. The next step is to deduce the Legendre biconjugate of $\Gamma$ by restricting the definition domain. Finally, to deal with points $\nu$ for which $l_\nu$ does not belong to the interior of $dom \Gamma^*$, we use a continuity argument similar to that of Theorem 7.1's proof in~\cite{leonard_1995}.\\

\indent The last result of this part gives an upper bound on the distance between a solution to \eqref{eq_theoapproxpop} and any $\nu$, this bound is used in Subsection~\ref{subsec:exittime}.
\begin{proposition}
 \label{prop_lowerboundI}
Let $T>0$ and $M>0$. There exists $C(T,M)$ such that, for any $\nu$ satisfying $\sup_{t\leq T} \langle \nu_t, \one \rangle <M$ and for all $(\xi_t)_{t\geq 0}$ solution to \eqref{eq_theoapproxpop} with the initial condition $\xi_0=\nu_0$, 
 \begin{equation*}
 \sup_{t\in [0,T]} \W(\nu_t,\xi_t) \leq C(T,M) \left(I^{T}(\nu)+\sqrt{I^{T}(\nu)}\right).
\end{equation*}  
 \end{proposition}

\begin{proof}
Let $\nu$ be such that $\langle\nu_t,\one \rangle <M$ for all $t\in [0,T]$. If $I^T(\nu)=0$, i.e. $\nu_t=\xi_t$, or if $I^{T}(\nu)=+\infty$, the result is obvious, so let us assume that $0<I^{T}(\nu)<+\infty$. Theorem~\ref{theo_functionalrate} implies the existence of $(h_1,h_2)\in L_{\rho^*,T} \times \LL^2_T$ such that
\begin{equation}
\label{eq_nonvar}
I^{T}(\nu)=\int_0^{T} \int_\CE \rho^*(h_1)d\mu^\nu_s ds +\int_0^{T} m\langle \nu_s,|h_2|^2 \rangle ds.
\end{equation}
We easily deduce that for any $t\leq T$,
\begin{equation}
\label{eq_minLL2}
 \|h_2\|_{\LL^2,t}^2 \leq \|h_2\|_{\LL^2,T}^2 \leq 2I^{T}(\nu). 
\end{equation}
Let us also find an upper bound on $\|h_1\|_{\rho^*,T}$. Note that for all $x\in \R$,
\begin{equation}
\label{eq_minorationrho*}
 \left\{\begin{aligned}
         &\text{if } \alpha \geq 1, && \rho^*(|x|/\alpha)\leq \rho^*(|x|)/\alpha \leq \rho^*(x)/\alpha,\\
	  & \text{if } 0< \alpha \leq 1, && \rho^*(|x|/\alpha)\leq \rho^*(|x|)/\alpha^2 \leq \rho^*(x)/\alpha^2.
        \end{aligned}
\right.
\end{equation}
Moreover, the non-variational formulation \eqref{eq_nonvar} implies that $I^T(\nu)\geq \int_0^{T} \int_\CE \rho^*(h_1)d\mu^\nu_s ds$. Thus, using
\eqref{eq_minorationrho*} and the definition of the norm $\|.\|_{\rho^*,T}$ in \eqref{def:normrho}, we obtain that
 if $I^T(\nu) \geq 1$, $ \int_0^{T} \int_\CE \rho^*\left({|h_1|}/{I^T(\nu)}\right)d\mu^\nu_s ds\leq 1$, i.e. $ \|h_1\|_{\rho^*,T} \leq I^T(\nu)$, and 
	 if $ I^T(\nu) \leq 1$, $\int_0^{T} \int_\CE \rho^*\left({|h_1|}/{\sqrt{I^T(\nu)}}\right)d\mu^\nu_s ds\leq 1$, i.e. $ \|h_1\|_{\rho^*,T} \leq \sqrt{I^T(\nu)}$.
Thus, for any $t\leq T$,
\begin{equation}
\label{eq_minrho*}
 \|h_1\|_{\rho^*,t} \leq  \|h_1\|_{\rho^*,T} \leq\left(I^{T}(\nu) + \sqrt{I^{T}(\nu)} \right) .
\end{equation}
Let $(\xi_t)_{t\geq 0}$ be the solution to \eqref{eq_theoapproxpop} with initial condition $\nu_0$. We want now evaluate $\W(\nu_t,\xi_t)$. Let us denote the semigroup of the reflected diffusion process which is the solution to \eqref{eq_brownien} with the initial condition $x$ and the diffusion coefficient $m^u$ by $(P_t^u)_{t\geq 0}$. Using Theorem~\ref{theo_functionalrate}, we find the following mild formulation for $(\nu_t)_{t\geq 0}$ in a similar way to Lemma 4.5 in \cite{champagnat_meleard_2007}~: for all $f \in C^{Lip}(\CX\times\CU)$,
\begin{equation}
 \langle \nu_t,f\rangle =  \int_0^t \int_\CE \psi(P^._{t-s}f)d\mu^\nu_sds +\int_0^t\int_\CE \psi(P^._{t-s}f) h_1 d\mu^\nu_sds +\int_0^t \langle \nu_s, 2 m \nabla_x P^._{t-s}f \cdot h_2 \rangle ds.
\end{equation}
In addition with a mild equation for $(\xi_t)_{t\geq 0}$, we deduce that for all $f \in C^{Lip}(\CX\times\CU)$ and for all $t\leq T$,
\begin{equation}
\label{eq_computing}
\begin{aligned}
 |\langle \nu_t-\xi_t,f\rangle| =& \bigg| \int_0^t \int_\CE \psi(P^._{t-s}f)(d\mu^\nu_sds-d\mu^\xi_sds) +\int_0^t\int_\CE \psi(P^._{t-s}f) h_1 d\mu^\nu_sds \\
&+\int_0^t \langle \nu_s, 2 m \nabla_x P^._{t-s}f \cdot h_2 \rangle ds \bigg|\\
\leq & C_1 \int_0^t \sup_{r\in [0,s]} \W(\nu_r,\xi_r)ds+ \|\psi(P^._{t-.}f)\|_{\rho,t} \|h_1\|_{\rho^*,t} +\|\nabla_xP_{t-.}^.f\|_{\LL^2,t} \|h_2\|_{\LL^2,t}.
\end{aligned}
\end{equation}
The second line is a consequence of Hölder's inequality (see for example Theorem 6 of Chapter 1 in \cite{rao_ren_2002} about Hölder's inequalities in Orlicz spaces).\\
Furthermore, the following Lemma insures that $P^u_tf \in C^{Lip}(\CX)$.
\begin{lemma}[Part 2 of \cite{wang_yan_2013}]
\label{lemma_lipschitz}
As $\CX$ is a convex set in $\R^d$, for all $f\in C^{Lip}(\CX)$, $u\in \CU$, and $t\in \R^+$, $P^u_tf \in C^{Lip}(\CX)$.
\end{lemma}
Let us now find an upper bound on $\|\psi(P_{t-.}^.f)\|_{\rho,t}$. $f$ belongs to $ C^{Lip}(\CX\times\CU)$ and $\sup_{t\in[0,T]}\langle \nu_t,\one\rangle\leq M$. So, for all $\alpha>0$, $t\leq T$
\begin{equation*}
\left| \int_0^{t}\int_{\CE}\rho\left(\frac{ |\psi(P^._{t-s}f)|}{\alpha}\right)d\mu^\nu_sds\right| \leq \left| \int_0^{T}\int_{\CE}\rho\left(\frac{ 1}{\alpha}\right)d\mu^\nu_sds\right|\leq T M [\bar b +\bar d +\bar c M] \rho\left( \frac{1}{\alpha} \right),
\end{equation*}
so, for any $t\leq T$,
\begin{equation*}
 \|\psi(P_{t-.}^. f)\|_{\rho,t} \leq \left[ (\rho_{|\R^+})^{-1} \left( \frac{1}{T M [\bar b +\bar d +\bar c M]}\right) \right]^{-1}:= C_2.
\end{equation*}
Furthermore, Lemma~\ref{lemma_lipschitz}
implies also that for any $t\leq T$,
\begin{equation*}
 \|\nabla_xP_{t-.}^.f\|_{\LL^2,t}^2 =\int_0^{t} 2\langle \nu_s, m |\nabla_x P^._{t-s}f|^2 \rangle ds \leq 2 \bar m M t \leq 2\bar m M T := C_3.
\end{equation*}
Using the last two inequalities with \eqref{eq_minLL2}, \eqref{eq_minrho*} and \eqref{eq_computing}, we find
\begin{equation*}
 \sup_{r\in [0,T]}\W(\nu_r,\xi_r) \leq C_1 \int_0^T \sup_{r\in [0,s]} \W(\nu_r,\xi_r)ds+ C_2 \left(I^{T}(\nu) + \sqrt{I^{T}(\nu)} \right) +  \sqrt{2 C_3 I^{T}(\nu)}.
\end{equation*}
We use Gronwall's Lemma to conclude.
\end{proof}

%%%%%%%%%%%%%%%%%%%%%%%%%%%%%%%%%%%%%%%%%%%%%%%%
%%%%%%%%%%%%%%%%%%%%%%%%%%%%%%%%%%%%%%%%%%%%%%%%%%
\section{Lower bound on the exit time of a neighborhood of the stationary state}
\label{sec_lowerbound}

\indent In this section, we assume that initially, two traits $u$ and $v$ are involved. The stochastic process starts in a state $\nu^K_0= \nu^{K,u}_0+\nu^{K,v}_0$ such that $\nu^{K,u}_0$ is close to $\bar\xi^u$ and there exist only a few individuals with trait $v$. Since the considered initial state is close to the equilibrium $(\bar\xi^u,0)$ and according to Theorem \ref{theo_limitpop}, the dynamics of the stochastic process $\nu^K$ is close to the equilibrium $(\bar\xi^u,0)$ on a finite interval time when $K$ is large. Our aim is to control the exit time of the stochastic process $\nu^{K,u}_t$ from a neighborhood of the stationary solution $\bar\xi^u$ in $M_F(\CX)$ when $K$ is large and $q_K$ is small. We define the exit time by~:
\begin{equation}
\label{def_RKgamma}
\text{for all } \gamma>0, R^K_{\gamma}=\inf\{t\geq 0, \W(\nu^{K,u}_t,\bar{\xi}^u)\geq \gamma)  \}.
\end{equation} 
Theorem~\ref{theo_exittimeintro} gives a lower bound on $R^K_\gamma$. The lower bound involves the first time when a new mutation occurs and the first time when the $v$-population size is larger than a threshold:
\begin{align}
\label{def_S1}
 &S_1^K=\inf\{t\geq 0, \exists w \not\in \{u,v\}, \nu^K_t(\CX \times\{w\}) \neq 0  \},\\
 \label{def_TKepsilon}
 &\text{ for all } \epsilon>0, T^K_\epsilon=\inf\{ t\geq 0, \langle\nu^{K,v}_t, \one \rangle \geq \epsilon \}.
\end{align}

\begin{theorem}
\label{theo_exittimeintro}
Suppose that Assumption~\ref{ass_coef} holds and that $H^u>0$. Let $\gamma>0$ such that $ \gamma < {H^u}{ (\kappa^{uu})}^{-1}$, and if $H^u\kappa^{vv}-H^v\kappa^{uv}>0$, $\gamma$ satisfies also the assumption $\gamma< |\frac{H^u}{\kappa^{uu}}-\frac{H^u\kappa^{vv}-H^v\kappa^{uv}}{\kappa^{uu}\kappa^{vv}-\kappa^{vu}\kappa^{uv}}|$ .
Then, there exist $\gamma'>0$, $\epsilon>0$, and $V>0$ such that, if $\nu^K_0=\nu^{K,u}_0+\nu^{K,v}_0$ with $\W(\nu^{K,u}_0,\bar\xi^u)<\gamma'$ and $\langle\nu^{K,v}_0,\one\rangle < \epsilon$, then
\begin{equation*}
\underset{K\to +\infty}{\lim} \; \Pb_{\nu^K_0}(R^K_{\gamma}>e^{KV} \wedge T^K_\epsilon \wedge S_1^K)=1.
\end{equation*}
\end{theorem}
Thus, a well-established monomorphic population $u$ is minimally affected during the emerging of a mutant population $v$. \\
\indent The assumptions on the radius $\gamma$ of the neighborhood ensure that there exists only one steady state in the neighborhood.\\
\indent The result is proved using ideas similar to the ones of Freidlin and Wentzell \cite{freidlin_wentzell_1984}. In our framework, the difficulties come from the continuous space motion. Firstly, our processes have values in an infinite dimensional space, thus, the required deterministic results are much more involved, see Subsection~\ref{subsec:stability}. 
Secondly, we deal with two kind of randomness~: jump process and spatial diffusion process.
The end of the section is devoted to the proof of Theorem~\ref{theo_exittimeintro}.

%%%%%%%%%%%%%%%%%%%%%%%%%%%%%%%%%%%%%%%%%%%%%%%%
%%%%%%%%%%%%%%%%%%%%%%%%%%%%%%%%%%%%%%%%%%%%%%%%%%

%%%%%%%%%%%%%%%%%%%%%%%%%%%%%%%%%%%%%%%%%%%%%%%%
%%%%%%%%%%%%%%%%%%%%%%%%%%%%%%%%%%%%%%%%%%%%%%%%%%
\subsection{Stability for the weak topology}
\label{subsec:stability}

\indent This subsection deals with the deterministic solution to \eqref{eq_theoapproxpop}. 
We denote by $(\xi_t)_{t\geq 0}$ the solution to equation \eqref{eq_theoapproxpop} with initial condition $\xi_0\in M_F(\CX\times\{u,v\})$. In this case, $\xi_t\in M_F(\CX\times \{u,v\})$ for all $t\geq 0$.
We prove that, as long as the size of the $v$-population density is small, the $u$-population density stays in a $\W$-neighborhood of its equilibrium $\bar g^u$. 
\begin{proposition}
\label{prop_stability}
Suppose that Assumption~\ref{ass_coef} holds. Let $\gamma>0$. There exist $\gamma'>0$ and $\epsilon'>0$ such that for any $\xi_0=\xi^u_0\delta_u +\xi^v_0\delta_v$ with $\W(\xi^u_0,\bar \xi^u)<  \gamma'$,
$$
\text{for all } t \leq t_{\epsilon'}=\inf\{t\geq 0, \langle \xi^v_t,\one \rangle > \epsilon'\}, \; \; \; \W(\xi^u_t,\bar \xi^u)<  \gamma/2.
$$
\end{proposition}

The proof of Proposition~\ref{prop_stability} implies two main difficulties.  First, using ideas similar to Part 3.3 of \cite{leman_mirrahimi_meleard_2014}, we can prove that the solution $\xi_t$ to \eqref{eq_theoapproxpop} stays close to $\bar \xi^u$ if the initial condition admits a density which is close to the density $\bar g^u$ of $\bar\xi^u$ for the $L^2$-distance. However, this is not sufficient since we will deal with discrete measures later. Thus, we need to enlarge the result for $\W$-distance. Secondly, we are concerned with the trajectories of the $u$-population process. Even though the $v$-population size is small, it does have an impact on the death rate of individuals $u$ which we cannot ignore.\\
The proof is divided into three steps. Firstly, we study how fast a solution with initial condition close to $\bar\xi^u$ moves away from $\bar\xi^u$ in $\W$-distance during a small time interval $[0,t_0]$. Then, as $t_0>0$, $\xi^u_{t_0}$ admits a density and so, we can compare the $\W$-distance and the $L^2$-distance of the densities between $\xi^u_{t_0}$ and $\bar \xi^u$. We finally prove a $L^2$-stability result similar to the one of Part 3.3 in~\cite{leman_mirrahimi_meleard_2014} but including the $v$-population process with a small size.\\

 \begin{proof}[Proof of Proposition~\ref{prop_stability}]
First we may assume that $\epsilon'\leq 1$ and $\gamma'<\gamma$. Hence, there exists $M>0$ such that any considered initial state satisfies $\langle \xi_0,\one \rangle <M$.\\
\indent We fix $t_0>0$ and we start with the first step. On the one hand, we can find an upper bound to $\sup_{r\in[0,t]} \langle \xi_r^u,\one \rangle$. Indeed
 \begin{equation}
 \label{eq_boundxiu}
  \langle \xi^u_t,\one \rangle \leq \langle \xi^u_0,\one \rangle +\bar b \int_0^t \langle \xi^u_s,\one \rangle ds,
 \end{equation}
 and using Gronwall's Lemma, we deduce that $\sup_{r\in[0,t]} \langle \xi^u_r,\one \rangle \leq M e^{\bar b t}$, for all $t\geq 0$.
On the other hand, using \eqref{eq_theoapproxpop} with $\CU=\{u,v\}$ and a mild formulation similar to Lemma 4.5 in \cite{champagnat_meleard_2007}, we find that $\xi^u_t$ satisfies: for any $f\in C^{Lip}(\CX)$,
 \begin{align*}
  \langle \xi^u_t- \bar\xi^u,f \rangle=& \langle \xi_0^u-\bar\xi^u ,P^u_tf \rangle + \int_0^t \langle \xi_s^u-\bar\xi^u, (b^u-d^u-c^{uu}\cdot\bar\xi^u )P^u_{t-s}f \rangle ds \\
  &+\int_0^t c^{uu}\cdot(\bar\xi^u-\xi^u_s) \langle \xi_s^u,  P^u_{t-s}f \rangle ds-\int_0^t (c^{uv}\cdot \xi^v_s) \langle \xi^u_s, P^u_{t-s}f \rangle ds .
 \end{align*}
For any $g$ Lipschitz-continuous, we denote by $\|g\|_{Lip}$ the smallest constant such that $g/\|g\|_{Lip} \in C^{Lip}(\CX)$. Since $\sup_{t\in [0,t_{\epsilon'}]} \langle \xi^v_t,\one \rangle \leq \epsilon'$, using Lemma~\ref{lemma_lipschitz} and the definition of distance $\W$, we obtain that, for all $t\leq t_{\epsilon'}$, 
\begin{equation}
 \label{eq_majeloignement}
 \begin{aligned}
  |\langle \xi_t^u-\bar\xi^u, f\rangle| \leq &\W(\xi_0^u,\bar\xi^u) +(\bar b + \bar d + \bar c \langle \bar \xi^u , \one \rangle )\int_0^t \W(\xi_s,\bar\xi)ds\\
  &+  \|c \|_{Lip} \sup_{r\in[0,t]} \langle \xi^u_r,\one \rangle \left( \int_0^t \W(\xi^u_s,\bar \xi^u)ds+\epsilon' \right).
 \end{aligned}
 \end{equation}
 Finally, \eqref{eq_boundxiu}, \eqref{eq_majeloignement} and Gronwall's Lemma imply that there exist $C_1, C_2$ independent of $\epsilon'$ and $\gamma'$ such that
 \begin{equation}
 \label{eq_debut}
  \sup_{r\in [0,t_0 \wedge t_{\epsilon'}]} \W(\xi_r^u,\bar\xi^u) \leq (\W(\xi_0^u,\bar\xi^u)+\epsilon' C_2) e^{C_1 t_0 \wedge t_{\epsilon'}}\leq  (\gamma'+\epsilon' C_2) e^{C_1 t_0}.
 \end{equation}
 According to \eqref{eq_debut}, we have to choose $\gamma'$ and $\epsilon'$ such that $(\gamma'+\epsilon' C_2) e^{C_1 t_0}<\gamma/2$.
 Note that if for all $\xi^v_0\in M_f(\CX)$, $t_{\epsilon'}\leq t_0$, the proof of Proposition~\ref{prop_stability} is complete. In what follows, let us assume that $t_{\epsilon'}>t_0$ for the considered initial state $\xi^v_0$.\\
\indent The next step is to compare the $L^2$-distance and the $\W$-distance between $\xi^u_{t_0}$ and $\bar\xi^u$.
According to Theorem~\ref{theo_limitpop}, for any $t_0>0$, $\xi^u_{t_0}$ has a Lipschitz-continuous density with respect to Lebesgue measure on $\CX$ that we denote by $g^u_{t_0}(x)$. In addition with the fact that $\bar g^u \in C^1(\CX)$, we have
 \begin{equation}
 \label{eq_l2W}
  \|g^u_{t_0}-\bar g^u\|^2_{L^2}=\int_\CX (g^u_{t_0}(x)-\bar g^u(x))^2dx \leq \W(\xi^u_{t_0},\bar\xi^u) (\|g^u_{t_0}\|_{Lip}+\|\bar g^u\|_{Lip}).
 \end{equation}
Let us bound $\|g^u_{t_0}\|_{Lip}$ from above. For any $t>0$, we define $h^u_{t}(x)=g^u_{t}(x) \exp(\int_0^{t_0}\left(c^{uu}\cdot g^u_s+c^{uv} \cdot \xi^v_s\right)ds)$. The exponent of the exponential term is positive and independent of $x$, thus $\|g^u_{t_0}\|_{Lip} \leq \|h^u_{t_0}\|_{Lip}$. Furthermore, according to Part 4 of chapter 5 in \cite{friedman_1964}, $h^u_{t_0}(x)=\int_\CX \Gamma_{t_0}(x,y)\xi_0(dy)$ where $\Gamma$ is the fundamental solution to the system
 \begin{equation*}
  \left\{
 \begin{aligned}
  &\p_t \Gamma =m^u\Delta \Gamma +(b^u(x)-d^u(x))\Gamma \text{ on } \CX \times \R^+,\\
 &\p_n \Gamma=0 \text{ on } \p \CX \times \R^+,\\
 & \Gamma(0,dx)=\xi_0(dx).
 \end{aligned}
 \right.
 \end{equation*}
As $t_0>0$, $\|\Gamma_{t_0}\|_{L^{\infty}(\CX)}$ and $\| \nabla \Gamma_{t_0}\|_{L^\infty(\CX)}$ are bounded from above and there exists $C_3$ such that
 \begin{equation}
 \label{eq_majlip}
  \|g^u_{t_0}\|_{Lip} \leq \|h^u_{t_0}\|_{Lip} \leq C_3 \langle \xi^u_0,\one \rangle \leq C_3 M,
 \end{equation}
 where $M$ has been defined in the beginning of the proof.
Combining \eqref{eq_debut}, \eqref{eq_l2W}, \eqref{eq_majlip} and the fact that $\W(\xi^u_0,\bar \xi^u)<\gamma'$, we find $C_4(\epsilon',\gamma')>0$ such that 
 \begin{equation*}
 \begin{aligned}
  \|g^u_{t_0}-\bar g^u\|_{L^2} &\leq \left((\W(\xi^u_0,\bar \xi^u )+C_2\epsilon')e^{C_1 t_0}(C_3M+\|\bar g^u\|_{Lip}) \right)^{1/2}\\ 
  &  \leq \left((\gamma'+C_2\epsilon')e^{C_1 t_0}(C_3 M+\|\bar g^u\|_{Lip}) \right)^{1/2} :=C_4(\epsilon', \gamma').
  \end{aligned}
 \end{equation*}
\indent  We now deal with the last step of the proof. 
We define $(\lambda_k,A_k)_{k\geq 1}$ the spectral basis for the operator $m^u \Delta_x.+(b^u-d^u).$ with Neumann boundary condition such that $(\lambda_k)_{k\geq 1}$ is a non-decreasing sequence with $H^u:=\lambda_1 > \lambda_2 \geq \lambda_3...$ and $(A_k)_{k\geq 1}$ is an orthonormal basis of $L^2(\CX)$. Note that $\bar g^u$ is equal to $\|\bar g^u \|_{L^2} A_1$. Let us express $g^u-\bar g^u$ in the basis $(A_k)_{k\geq 1}$
\begin{equation*}
  g^u_t(x)=\bar g^u(x) +\sum_{i=1}^{+\infty} \alpha_i(t)A_i(x),
 \end{equation*} 
 and denote for all $t\in \R^+$,
 \begin{equation*}
 \beta(t):=1+\frac{\alpha_1(t)}{\|\bar g^u \|_{L^2}}.
 \end{equation*}
From \eqref{eq_theoapproxpop} and the representation of $g^u-\bar g^u$ and $\partial_t g^u$ with respect to the basis $(A_k)_{k\geq 1}$, we find the following dynamical system 
 \begin{align}
  \label{eq_systeme1}
&\p_t \alpha_k(t)=\alpha_k(t) \left( \lambda_k-H^u-\int_{\CX}c^{uv} g^v_t-\int_\CX c^{uu}(g^u_t-\bar g^u) \right), \text{ for all } k\geq 2,\\
\label{eq_systeme2}
&\p_t \beta(t)=\beta(t)\left( H^u-\int_\CX c^{uv}g^v_t-\sum_{i=2}^{+\infty} \alpha_i(t)\Big(\int_\CX c^{uu}A_i\Big) -H^u \beta(t) \right).
  \end{align}
 The last step of the proof consists in proving that if $\epsilon'$ and $\gamma'$ are sufficiently small such that
 \begin{equation}
 \label{eq_hypepsgamma}
  \max\left\{1, \frac{\|\bar g^u\|_{L^2}}{H^u} \right\} \cdot (\epsilon' \bar c +3\|c^{uu}\|_{L^2} C_4(\epsilon',\gamma') ) < \min\left\{\frac{H^u-\lambda_2}{2},\frac{\gamma}{2\|\one\|_{L^2}}\right\},
 \end{equation}
then for all $t\leq t_{\epsilon'}$, $\W(g^u_t,\bar g^u)< \gamma/2$.
Let us fix $K>0$ such that 
\begin{equation}
 \label{eq_defK}
\epsilon' \bar c +3\|c^{uu}\|_{L^2} C_4(\epsilon',\gamma')\leq K/2 < (H^u-\lambda_2)/2, 
\end{equation}
%We will show that for all $t\leq t_{\epsilon'}$, $|\int_\CX c^{uu}(g^u_t-\bar g ^u)|\leq K.$
%Let us suppose that 
and
$$
t_{max}=\inf\{t\geq t_0, |\int_\CX c^{uu}(g^u_t-\bar g ^u)|\geq K\}. 
$$
Let us prove that $t_{max}\geq t_{\epsilon'}$. 
At $t=t_0$, $|\int_\CX c^{uu}(g^u_{t_0}-\bar g^u)|\leq \|c^{uu}\|_{L^2}C_4(\epsilon',\gamma')\leq K/2<K.$ A continuity argument implies that $t_{max}>t_0$.
Then, using \eqref{eq_systeme1}, we deduce that for any $k\geq 2$, for any $t_0\leq t\leq t_{max}\wedge t_{\epsilon'}$,
$$
\alpha_k(t)^2\leq \alpha(t_0)^2 \exp(2(\lambda_k-H^u+K)(t-t_0))\leq \alpha(t_0)^2.
$$
Thus 
\begin{equation}
 \label{eq_sumalphai}
\begin{aligned}
 \left|\sum_{i=2}^{+\infty} \alpha_i(t)\int_\CX (c^{uu}A_i)\right|&\leq \left(\sum_{i=2}^{+\infty}\alpha_i(t)^2 \right)^{1/2}\left(\sum_{i=2}^{+\infty}(\int_\CX c^{uu}A_i)^2 \right)^{1/2}\\
 &\leq \left(\sum_{i=1}^{+\infty}\alpha_i(t_0)^2 \right)^{1/2}\left(\sum_{i=1}^{+\infty}(\int_\CX c^{uu}A_i)^2 \right)^{1/2} \leq C_4(\epsilon',\gamma') \|c^{uu}\|_{L^2}.
\end{aligned}
\end{equation}
Inserting \eqref{eq_sumalphai} in Equation \eqref{eq_systeme2} implies that, for all $t\leq t_{max}\wedge t_{\epsilon'} $,
\begin{multline*}
\beta(t) \left(H^u-\epsilon' \bar c -C_4(\epsilon',\gamma')\|c^{uu}\|_{L^2}-H^u\beta(t)\right)\leq \p_t\beta(t) \\
\leq \beta(t) \left(H^u+\epsilon' \bar c +C_4(\epsilon',\gamma')\|c^{uu}\|_{L^2}-H^u\beta(t)\right).
\end{multline*}
Moreover, $|\beta(t_0)-1|\|\bar g^u\|_{L^2} \leq C_4(\epsilon',\gamma')$. Using properties of logistic equations, we deduce that for all $t_0\leq t\leq t_{max} \wedge t_{\epsilon'}$,
\begin{equation}
 \label{eq_alpha1}
|\alpha_1(t)| = |\beta(t)-1|\|\bar g^u\|_{L^2}  \leq \frac{\|\bar g^u\|_{L^2}}{H^u}(\epsilon' \bar c +C_4(\epsilon',\gamma')\|c^{uu}\|_{L^2})+C_4(\epsilon',\gamma').
\end{equation}
Thus, for all $t_0\leq t\leq t_{max}\wedge t_{\epsilon'}$, \eqref{eq_sumalphai} and \eqref{eq_alpha1} imply that
\begin{align*}
\left|\int_\CX c^{uu}(g^u_t-\bar g ^u)\right|&\leq |\alpha_1(t)|\frac{H^u}{\|\bar g^u\|_{L^2}}+\left| \sum_{i=2}^{+\infty} \alpha_i(t)\int_\CX (c^{uu}A_i) \right|\\
&\leq \epsilon' \bar c +2C_4(\epsilon', \gamma')\|c^{uu}\|_{L^2} + \frac{H^u}{\|\bar g^u\|_{L^2}} C_4(\epsilon', \gamma').
\end{align*}
By definition, $H^u=\int_\CX c^{uu} \bar g^u\leq \|c^{uu}\|_{L^2} \|\bar g^u\|_{L^2}$ and \eqref{eq_defK} gives for all $t\leq t_{max} \wedge t_{\epsilon'}$,
\begin{equation}
\label{eq_K2}
\left|\int_\CX c^{uu}(g^u_t-\bar g ^u)\right|\leq \epsilon' \bar c +3C_4(\epsilon', \gamma')\|c^{uu}\|_{L^2}\leq \frac{K}{2}.
\end{equation}
Inequality~\eqref{eq_K2} implies that $t_{max}\wedge t_{\epsilon'} < t_{max}$, that is, $t_{max}\wedge t_{\epsilon'} = t_{\epsilon'}$ and all inequalities proved above are true for all $t\leq t_{\epsilon'}$. In addition, \eqref{eq_hypepsgamma}, \eqref{eq_sumalphai} and \eqref{eq_alpha1} imply that, for all $t\leq t_{\epsilon'}$,
\begin{align*}
\W(g^u_t,\bar g^u)&\leq \|g^u_t-\bar g^u\|_{L^2}\|\one\|_{L^2} \leq \|\one\|_{L^2}\left( |\alpha_1(t)|+|\sum_{i=1}^{+\infty} \alpha_i(t)^2 |^{1/2} \right)\\
&\leq \frac{\|\one\|_{L^2}\|\bar g^u\|_{L^2}}{H^u}(\epsilon'\bar c +3C_4(\epsilon', \gamma')\|c^{uu}\|_{L^2})\leq \frac{\gamma}{2}.
\end{align*}
  That ends the proof.
 \end{proof}

%%%%%%%%%%%%%%%%%%%%%%%%%%%%%%%%%%%%%%%%%%%%%%%%
%%%%%%%%%%%%%%%%%%%%%%%%%%%%%%%%%%%%%%%%%%%%%%%%%%
\subsection{Exit time}
\label{subsec:exittime}

This subsection is devoted to the proof of Theorem~\ref{theo_exittimeintro}. We split the proof into three lemmas similar to the ones in Dembo and Zeitouni \cite{dembo_zeitouni_1998}.\\
Let $\gamma>0$ satisfying the assumptions of Theorem~\ref{theo_exittimeintro}. 
We consider $\epsilon'$ and $\gamma'$ as in Proposition~\ref{prop_stability} and set $\epsilon=\frac{\epsilon'}{2}$ and $\rho=\frac{\gamma'}{3}<\gamma$. $R^K_\gamma$, $S_1$ and $T^K_\epsilon$ have been defined by \eqref{def_RKgamma}, \eqref{def_S1} and \eqref{def_TKepsilon} and
let us define 
$$
\tau=\inf\{ t\geq 0, \W(\nu_t^{K,u},\bar{\xi}^u)\not\in ]\rho,\gamma[ \}.
$$

\begin{lemma}
\label{lemma_exit1}
Under Assumption~\ref{ass_coef},
we have
\begin{equation*}
\underset{t\to +\infty}{\lim} \; \underset{K \to +\infty}{\limsup} \; \dfrac{1}{K} \log \underset{\nu_0^K \in (B(\bar{\xi}^u,\gamma)\times B(0,\epsilon))\cap M_F^K}{\sup}\; \Pb_{\nu^K_0}(\tau \wedge T^K_\epsilon \wedge S_1^K >t)=-\infty.
\end{equation*}
\end{lemma}
\noindent
That is, the probability that the process $\nu^{K,u}$ stays a long time in the ring $B(\bar{\xi}^u,\gamma) \setminus B(\bar{\xi}^u,\rho)$ is exponentially small. The proof requires a comparison with the deterministic paths of Equation~\eqref{eq_theoapproxpop}, the difficulty is to prove that there exists a finite time after which all deterministic paths starting in the ring are out of the ring. The fact that the probability is exponentially small is a consequence of Theorem~\ref{theo_majoration}.

\begin{lemma}
\label{lemma_exit2}
Under Assumption~\ref{ass_coef}, there exists $V>0$ such that 
\begin{equation*}
 \underset{K \to +\infty}{\limsup} \; \dfrac{1}{K} \; \log \underset{\nu_0^K \in (B(\bar{\xi}^u,\gamma')\times B(0,\epsilon))\cap M_F^K}{\sup}\; \Pb_{\nu^K_0}(\W(\nu^{K,u}_{\tau },\bar{\xi}^u)\geq \gamma, \tau \leq T^K_\epsilon \wedge S_1^K)\leq -V.
\end{equation*}
\end{lemma}
\noindent
%That is, the probability that the process goes out of $ B(\bar{\xi}^u,\gamma)$ before entering in $B(\bar{\xi}^u,\rho)$ is exponentially small. 
Once again, the proof is based on a comparison with the deterministic paths, it is a consequence of Proposition~\ref{prop_stability} and Theorem~\ref{theo_majoration}.

\begin{lemma}
\label{lemma_exit4}
Under Assumption~\ref{ass_coef}, for all $C>0$, there exists $T(C,\rho)>0$ such that,
\begin{equation*}
\underset{K\to +\infty}{\limsup} \; \dfrac{1}{K}\; \log \underset{\nu^{K,u}_0 \in B(\bar{\xi}^u,\gamma)\cap M^K_F}{\sup} \; \Pb_{\nu^K_0}\bigg(\underset{t\in [0,T(C,\rho)]}{\sup}\; \W(\nu^{K,u}_t,\nu^{K,u}_0)\geq \rho\bigg)<-C. 
\end{equation*}
\end{lemma}
\noindent
This lemma means that there exists a small time interval during which any process stays close from its starting point with an exponentially high probability. The stochastic process includes jump parts and diffusive parts. Thus, we study not only the size of the population process during a small time interval, but also a sum of reflected diffusion processes.\\

Theorem~\ref{theo_exittimeintro} is proved using the two last lemmas \ref{lemma_exit2} and \ref{lemma_exit4}. We do not detail its proof as it can be adapted from \cite{dembo_zeitouni_1998}. The main change is that the proof has to be done on $\{t\leq T^K_\epsilon \wedge S_1^K\}$ to ensure that the $v$-process size $\langle \nu^{K,v},\one \rangle$ is small and that no other mutation appears, but Lemma~\ref{lemma_exit2} is sufficient to circumvent this difficulty.
%Let us just give some ideas.
%Two infinite dimensional balls $B(\bar\xi^u,\rho)$ and $B(\bar\xi^u,2\rho)$ are used, and one studies the walk in and outside those balls. Lemma~\ref{lemma_exit2} gives the initial behavior of the process between its starting point and the first time when it enters $B(\bar\xi^u,\rho)$. Then using the Markov property, one studies a process that starts in $B(\bar\xi^u,\rho)$. In fact, as a consequence of Lemma~\ref{lemma_exit2}, this process goes back and forth a lot of times between the exterior of $B(\bar\xi^u,2\rho)$ and the interior of $B(\bar\xi^u,\rho)$ before going out of the domain $B(\bar\xi^u,\gamma)$ and as long as $t\leq T^K_\epsilon \wedge S_1^K$. Finally Lemma~\ref{lemma_exit4} introduces a minimum time to go from $B(\bar\xi^u,\rho)$ to $B^c(\bar\xi^u,2\rho)$ that allows us to conclude.

\noindent
There remains to prove the three lemmas.
%%%%%%%%%%%%%%%%%%%%%%%%%%%%%%%%%%%%%%%%%%%%%%%%%%%%%%%%
\begin{proof}[Proof of Lemma~\ref{lemma_exit1}]
 Remark that if $\W(\nu^{K,u}_0,\bar\xi^u)< \rho$ for all $K\geq 1$, the result is obvious.
 Otherwise, let us define the following set:
 \begin{equation}
 \label{eq_defA}
 \mathcal{A}(T)=adh \{ \nu\in \D([0,T],M_F(\CX \times \{u,v\})), \forall t \in [0,T], \W(\nu^u_t,\bar\xi^u)\in ]\rho,\gamma[ \text{ and } \langle\nu^v_t,\one \rangle < \epsilon \},
 \end{equation}
 where $adh S$ is the closure of the set $S$.
 Remark that $\{\nu^K\in \mathcal{A}(T)\}=\{\tau \wedge T^K_\epsilon \wedge S_1^K  > T\}$ a.s. and that the set $\CC=adh( B(\bar\xi^u, \gamma) \times B(0,\epsilon))$ is a compact set of $M_F(\CX\times \{u,v\})$ as $\CX$ is bounded. By applying Theorem~\ref{theo_majoration} with the closed set $\mathcal{A}(T)$ and the initial compact set $\CC$, we find
 \begin{equation*}
 \underset{K \to +\infty}{\limsup} \; \dfrac{1}{K} \log \underset{\nu_0^K \in \CC\cap M_F^K}{\sup}\; \Pb_{\nu^K_0}(\nu\in \mathcal{A}(T))\leq -\underset{\nu \in \mathcal{A}(T)}{\inf} I^T(\nu).
 \end{equation*}
 Thus, the lemma will be proved if we show that ${\inf}_{\nu \in \mathcal{A}(T)} I^T(\nu) {\rightarrow} +\infty$ as $T$ tends to $+\infty$. 
 To this aim, we will first show that any solution to \eqref{eq_theoapproxpop} cannot belong to $\mathcal{A}(T_0)$ if $T_0$ is large enough. Precisely, we set $\delta \in ]0, \rho/2[$, and we will prove that there exists $T_0>0$, such that any $(\xi_t)_{t\geq 0}$ solution to \eqref{eq_theoapproxpop} with an initial condition satisfying $\W(\xi^u_0,\bar\xi^u)\in]\rho,\gamma[$ and $\langle \xi^v_0,\one \rangle <\epsilon$ satisfies
 \begin{equation}
\label{eq_minnorm}
 \underset{\nu\in \mathcal{A}(T_0)}{\inf} \underset{t\in[0,T_0]}{\sup} \W(\xi_t,\nu_t)\geq \delta.
 \end{equation}

 %%%%%%%%%%%%%%%%%%%%%%%%%%%%%%%%%%%%%%%%%%%%%%%%%%%%%
\noindent Assume that \eqref{eq_minnorm} holds. Since $\sup_{\nu\in \mathcal{A}(T_0)} \sup_{t\in [0,T_0]} \langle \nu,\one \rangle$ is bounded, we can use Proposition~\ref{prop_lowerboundI} to deduce that there exists $C>0$ such that for any $\nu\in \mathcal{A}(T_0)$, $\delta \leq C (I^{T_0}(\nu)+\sqrt{I^{T_0}(\nu)})$. As $x \mapsto x+\sqrt{x}$ is a bijective function from $\R^+$ to $\R^+$, we find a constant $C(\delta)>0$ which is a lower bound on $I^{T_0}(\nu)$ and 
$
\inf_{\nu\in \mathcal{A}(T_0)} I^{T_0}(\nu)\geq C(\delta).
$
 Finally, for $T>T_0$ and $\nu\in \mathcal{A}(T)$, we decompose $\nu$ as a sum of $n:=[T/T_0]$ terms~: using Chasles' Theorem and the non-variational formulation~\eqref{eq_theorate} of $I^T>0$, we find a sequence $(\nu^i)_{i=1..n}\in \mathcal{A}(T_0)$ such that 
 \begin{equation*}
 I^T(\nu)\geq \sum_{i=1}^n I^{T_0}(\nu^i) \geq C(\delta) n \underset{T \rightarrow +\infty}{\longrightarrow} +\infty.
 \end{equation*}

\indent
 There remains to prove that \eqref{eq_minnorm} holds.
Let $\W(\xi_0^u,\bar\xi^u)\in ]\rho,\gamma[$ and $\langle \xi^v_0,\one\rangle <\epsilon$, and let $(\xi_t)_{t\geq 0}$ be the solution to \eqref{eq_theoapproxpop} with initial condition $\xi_0=\xi^u_0\delta_u+\xi^v_0\delta_v$. Using Theorems 1.2 and 1.4 in \cite{leman_mirrahimi_meleard_2014}, we know that $\xi_t\in M_F(\CX\times\{u,v\})$ converges to a stationary state. Theorem~1.2 in~\cite{leman_mirrahimi_meleard_2014} implies that this stationary state is either $\bar\xi^u\delta_u$ or $\bar\xi^v \delta_v$ or a state with coexistence but the assumptions on $\gamma$ insure that those stationary states do not belong to $(B(\bar\xi^u,\gamma) \setminus B(\bar\xi^u,\rho)) \times B(0,\epsilon)$. Therefore, there exists $T_{\xi_0}$ such that 
 \begin{equation}
 \label{eq_Tdeter}
 \W(\xi^u_{T_{\xi_0}},\bar\xi^u)\not\in [ \gamma+2\delta, \rho-2\delta ] \text{ or } \langle\xi^v_{T_{\xi_0}},\one \rangle \geq \epsilon+2\delta.
 \end{equation}
However $T_{\xi_0}$ depends on $\xi_0$. Thus, we will use a compactness argument to find a uniform time and conclude.
 First, using Gronwall's Lemma, we obtain~: for $M\geq \langle \bar\xi^u,\one\rangle+2\gamma+2\epsilon$,
 \begin{equation}
 \label{eq_boundxione}
 \sup_{t\in[0,T_{\xi_0}]} \langle \xi_t, \one \rangle \leq \langle \xi_0, \one \rangle e^{\bar b T_{\xi_0}} \leq M e^{\bar b T_{\xi_0}}.
 \end{equation}
Then we show that two solutions $(\xi_t)_{t\geq 0}$ and $(\zeta_t)_{t\geq 0}$ to \eqref{eq_theoapproxpop} which are close initially stay close during a short time. Indeed, using the mild equation of Lemma 4.5 in~\cite{champagnat_meleard_2007}, for any $f \in C^{Lip}(\CX\times\{u,v\})$ 
\begin{equation*}
\langle \xi_t-\zeta_t, f \rangle =\langle \xi_0-\zeta_0,P_{t}f \rangle  + \int_0^t \langle \xi_s-\zeta_s,(b-d-c\cdot \xi_s)P_{t-s}f \rangle ds -\int_0^t \langle \xi_s, c\cdot(\xi_s-\zeta_s)P_{t-s}f\rangle ds.
\end{equation*}
 According to Lemma~\ref{lemma_lipschitz} and Assumptions~\ref{ass_coef}, %on the Lipschitz properties of $b$, $d$ and $c$,
 we can find a constant $C_1$ such that $(b-d-c\cdot \xi_t)P_{t-s}f/C_1$ and $c/C_1$ belong to $C^{Lip}(\CX\times \{u,v\})$. We deduce that for all $t\leq T_{\xi_0}$,
 \begin{equation}
\label{eq_minxizeta}
| \langle \xi_t-\zeta_t, \varphi \rangle| \leq \W(\xi_0,\zeta_0)  +C_1 \int_0^t \sup_{r\in[0,s]}\W( \xi_r,\zeta_r) ds +C_1 \sup_{r\in [0,T_{\xi_0}]}\langle \xi_r,\one \rangle  \int_0^t  \sup_{r\in [0,s]}\W(\xi_r,\zeta_r)ds .
 \end{equation}
 Using Gronwall's Lemma and \eqref{eq_boundxione}, we conclude that there exists a constant $C(T_{\xi_0})$ such that
 \begin{equation*}
 \sup_{r\in [0,T_{\xi_0}]}\W(\xi_r,\zeta_r) \leq C(T_{\xi_0}) \W(\xi_0,\zeta_0).
 \end{equation*}
 Choosing $\alpha_{\xi_0}=\delta/C(T_{\xi_0})$, we find for all $\zeta_0$ with $\W(\xi_0,\zeta_0)<\alpha_{\xi_0}$, $\sup_{r\in [0,T_{\xi_0}]}\W(\xi_r,\zeta_r)< \delta$. In addition with \eqref{eq_Tdeter} and~\eqref{eq_defA}, we find that for all $\zeta_0$ with $\W(\zeta_0,\xi_0)< \alpha_{\xi_0}$,
 \begin{equation*}
 \sup_{t\in [0,T_{\xi_0}]} \sup_{\nu\in \mathcal{A}(T_0)}\W(\zeta_t,\nu_t) \geq \delta.
 \end{equation*}
 Remark that $adh(( B(\bar\xi^u,\gamma)\setminus B(\bar\xi^u,\rho))\times B(0,\epsilon))$ is a compact set of $M_F(\CX\times \{u,v\})$ as $\CX$ is bounded. It is covered by $\cup_{\xi_0 \in (B(\bar\xi^u,\gamma)\setminus B(\bar\xi^u,\rho)\times B(0,\epsilon))} B(\xi_0,\alpha_{\xi_0})$. We extract a finite cover
$
 \cup_{i=1}^n B(\xi_0^i,\alpha_{\xi_0^i}).
 $
 Finally, defining $T_0=\max_{i=1..n} T_{\xi_0^i}$, we conclude~: for any $\xi_0$ with $\W(\xi_0^u,\bar\xi^u)\in ]\rho,\gamma[$ and $\langle\xi_0^v,\one\rangle <\epsilon$, we have
 \begin{equation*}
 \underset{\nu\in \mathcal{A}(T_0)}{\inf} \underset{t\in[0,T_0]}{\sup} \W(\xi_t,\nu_t) \geq \delta.
 \end{equation*}
%%%%%%%%%%%%%%%%%%%%%%%%%%%%%%%%%%%%%%%%%%%%%%%%%%%%%%%%%%%%%%

\noindent Proof of Lemma~\ref{lemma_exit1} is now complete.
\end{proof}

%%%%%%%%%%%%%%%%%%%%%%%%%%%%%%%%%%%%%%%%%%%%%%%%%%%%%%%%%%%%%%

\begin{proof}[Proof of Lemma~\ref{lemma_exit2}]
Lemma~\ref{lemma_exit1} gives $T_1$ such that 
\begin{equation}
\label{eq_inegalite1}
 \underset{K \to +\infty}{\limsup} \; \dfrac{1}{K} \log \underset{\nu_0^K \in (B(\bar{\xi}^u,\gamma)\times B(0,\epsilon))\cap M_F^K}{\sup}\; \Pb_{\nu^K_0}(\tau \wedge T^K_\epsilon \wedge S_1^K >T_1)\leq -1.
\end{equation}
Thus we limit our study to the time interval $[0,T_1]$. Since the considered initial states satisfy $\langle \nu^K_0 ,\one \rangle \leq (\langle \bar\xi^u ,\one \rangle +2\gamma+2\epsilon)$, using Lemma~\ref{lemma_tensionexpo}, we find $N>0$ such that
\begin{equation}
\label{eq_inegalite2}
 \underset{K \to +\infty}{\limsup} \; \dfrac{1}{K} \log \underset{\nu_0^K \in (B(\bar{\xi}^u,\gamma)\times B(0,\epsilon))\cap M_F^K}{\sup}\; \Pb_{\nu^K_0}(\sup_{t\in[0,T_1]}\langle \nu_t^K,\one \rangle \geq N)\leq -1.
\end{equation}
Let $M\geq (\langle \bar\xi^u ,\one \rangle +2\gamma+2\epsilon)\vee N$ and
\begin{multline}
\label{def_A}
 \mathcal{A}=\big\{ \nu\in \D([0,T_1],M_F(\CX\times \{u,v\}))\; |\;  \exists t \in [0,T_1], \W(\nu_t^u,\bar{\xi}^u)< \gamma, \sup_{t\in [0,T_1]} \langle\nu^v_t,\one \rangle < \epsilon  \\
  \text{ and } \sup_{t\in[0,T_1]} \langle\nu_t,\one\rangle < M \big\}.
\end{multline}
For all $\nu^K_0$, and $K$,
\begin{align*}
 \Pb_{\nu^K_0}(&\W(\nu^{K,u}_{\tau},\bar{\xi}^u) \geq  \gamma, \tau \leq T^K_\epsilon \wedge S_1^K) \\
 &\leq \Pb_{\nu^K_0}(\tau \leq T_1, \sup_{t\in[0,T_1]}\langle\nu^K_t,\one \rangle \geq M)+\Pb_{\nu^K_0}(\tau \leq T_1, \nu^K\in \mathcal{A})+\Pb_{\nu^K_0}(\tau>T_1, \tau \leq T^K_\epsilon\wedge S_1^K)\\
&\leq \Pb_{\nu^K_0}(\sup_{t\in[0,T_1]}\langle\nu^K_t,\one \rangle \geq M)+\Pb_{\nu^K_0}( \nu^K\in \mathcal{A})+\Pb_{\nu^K_0}(\tau \wedge T^K_\epsilon\wedge S_1^K>T_1).
\end{align*}
Then, we use Theorem~\ref{theo_majoration}, the definition of $\mathcal{A}$~\eqref{def_A}, \eqref{eq_inegalite1}, and \eqref{eq_inegalite2} to find
\begin{equation*}
\underset{K\to +\infty}{\limsup} \; \dfrac{1}{K} \; \log \underset{\nu^K_0 \in \CC \cap M^K_F}{\sup}\;\Pb_{\nu^K_0}(\nu^K_{\tau}\in B^c(\bar{\xi}^u,\gamma))  \leq \max\left\{ -1,-\inf_{\nu\in \mathcal{A}, \nu_0\in \CC} I^{T_1}(\nu) \right\},
\end{equation*}
where $\CC=adh(B(\bar{\xi}^u,3\rho)\times B(0,\epsilon))$. There remains to prove that the r.h.s. is strictly negative. Proposition~\ref{prop_stability} implies that any solution $(\xi_t)_{t\geq 0}$ to \eqref{eq_theoapproxpop} with $\xi_0\in \CC$ satisfies the property~:
$$
\text{if } \sup_{t\in[0,T_1]} \langle \xi^v_t,\one\rangle < \epsilon'=2\epsilon, \; \text{ then }  \W(\xi^u_t,\bar\xi^u) < \gamma/2.
$$
We deduce immediately that for any $\nu\in \mathcal{A}$, with $\nu_0\in \CC$, if $\xi$ is the solution to \eqref{eq_theoapproxpop} with $\xi_0 :=\nu_0$,
$$
\sup_{t\in [0,T_1]}\W(\nu_t,\xi_t)\geq \sup_{t\in[0,T_1]}\max\{ \W(\nu_t^u,\xi_t^u), \W(\nu_t^v,\xi_t^v)  \} \geq \frac{\gamma}{2} \wedge \epsilon.
$$
We conclude the proof using Proposition~\ref{prop_lowerboundI} and a compactness argument similar to the one at the end of the previous proof.
\end{proof}

%%%%%%%%%%%%%%%%%%%%%%%%%%%%%%%%%%%%%%%%%%%%%%%%%%%%%%%%
\begin{proof}[Proof of Lemma~\ref{lemma_exit4}]
 Let us fix $f\in C^{Lip}(\CX)$ and study the following difference, using the construction of the process $\nu^{K,u}$:
\begin{eqnarray*}
 |\langle \nu^{K,u}_{t},f\rangle -\langle \nu^{K,u}_{0}, f \rangle | 
 \leq  \dfrac{1}{K} \left( \sum_{i \in N^{\text{notdead}}_t} |f(X^i_{t})-f(X^i_0)| +\sum_{i \in N^{\text{dead}}_t} |f(X^i_0)| +\sum_{i\in N^{\text{born}}_t} |f(X^i_t)| \right),
\end{eqnarray*}
where $N^{\text{notdead}}_t$ is the set of indices of the individuals with trait $u$ alive at time $0$ and not dead during $[0,t]$; $N^{\text{dead}}_t$ represents the set of indices of the individuals with trait $u$ alive at time $0$ and dead during $[0,t]$; and $N^{\text{born}}_t$ is the set of indices of the individuals with trait $u$ born during $]0,t]$. As $f\in C^{Lip}(\CX)$,
\begin{eqnarray*}
|\langle \nu^{K,u}_{t},f\rangle -\langle \nu^{K,u}_{0}, f \rangle | %&\leq &  \dfrac{1}{K} \left( \sum_{i \in N^{\text{notdead}}_t} |X^i_{t}-X^i_0| +\sum_{i \in N^{\text{dead}}_t} 1 +\sum_{i\in N^{\text{born}}_t} 1 \right),\\
 &\leq & \dfrac{1}{K} \left( \sum_{i \in N^{\text{notdead}}_t} |X^i_{t}-X^i_0|\right) +\dfrac{N^{\text{dead}}_t+N^{\text{born}}_t}{K}.
\end{eqnarray*}
As this is true for all $f\in C^{Lip}(\CX)$, we find the same upper bound for $\W(\nu_t^{K,u},\nu_0^{K,u})$.
We use now the stopping time 
$$
\tau^K_N=\inf\{ t>0, |\langle \nu^{K,u}_t,\one \rangle| >N \}.
$$
On $\{\tau_N^K\geq t\}$, $\W(\nu_t^{K,u},\nu_0^{K,u})$ is stochastically bounded by
$
\frac{1}{K} \sum_i^{KN} |X^i_t-X^i_0|+\frac{\mathcal{P}(t)}{K},
$
where $\{(X^i_t)_{t\geq 0}\}_{i\in\{1..KN\}}$ are $KN$ independent reflected diffusion processes driven by Equation~\eqref{eq_brownien} with the diffusion coefficient $m^u$ and $(\mathcal{P}(t))_{t\geq 0}$ is a Poisson process with intensity $(\bar{b}+\bar{d}+N\bar{c})KN$. Finally,
\begin{equation}
\label{eq_3termesbis}
\begin{aligned}
\Pb\Big(\underset{t\in[0,T]}{\sup}& \W(\nu^{K,u}_t,\nu^{K,u}_0) \geq \rho\Big)  \\
& \leq  \Pb(\tau^K_N\leq T)+\Pb\Big(\tau_N^K \geq T, \underset{t\in[0,T]}{\sup} \W(\nu^{K,u}_t,\nu^{K,u}_0)\geq  \rho\Big),\\
&\leq  \Pb(\tau^K_N\leq T)+\Pb\Big(\frac{1}{K}\sum_{i=0}^{KN}\underset{t\in[0,T]}{\sup} |X^i_t-X^i_0|\geq \frac{\rho}{2}\Big)+\Pb\Big( \underset{t\in[0,T]}{\sup} \frac{\mathcal{P}(t)}{K} \geq \frac{\rho}{2}\Big),\\
%&\leq  \Pb(\tau^K_N\leq T)+\Pb\Big( \sum_{i=0}^{KN}\underset{s\in[0,1]}{\sup} \sqrt{T} |B^i_s|\geq \frac{K\rho}{2}\Big)+\Pb\Big(  \mathcal{P}(T) \geq \frac{K\rho}{2}\Big),\\
%&\leq \Pb(\tau^K_N\leq T)+\Pb\Big( \prod_{i=0}^{KN}\underset{s\in[0,1]}{\sup} e^{|B^i_s|}\geq e^{\frac{K\rho}{2\sqrt{2 m^u T}}}\Big)+\Pb\Big( e^{r \mathcal{P}(T)} \geq e^{r\frac{K\rho}{2}}\Big), \; \forall r\geq 0.
\end{aligned}
\end{equation}
Using Lemma~\ref{lemma_tensionexpo}, we can fix $N\in \N$ such that 
\begin{equation*}
\underset{K\to +\infty}{\limsup} \frac{1}{K}\log\Pb(\tau^K_N\leq T)\leq -C.
\end{equation*}
\noindent
Let us now consider the second term of \eqref{eq_3termesbis}. 
For $\bar x=(x^1,..,x^{KN}) \in \CX^{KN}$, we denote the probability under which $(X^1_0,..,X^{KN}_0)$ is equal to $\bar x$ by $\Pb_{\bar x}$. 
Let $\Upsilon$ be the stopping time $\Upsilon=\inf\{s\geq 0, \sum_{i=0}^{KN} |X^i_s-X^i_0|\geq \rho K/2 \}$. Using the Markov property, we find
\begin{equation}
\label{eq_sup}
 \begin{aligned}
  \Pb_{\bar x }\Big(&\underset{t\in[0,T]}{\sup} \sum_{i=0}^{KN} |X^i_t-X^i_0|\geq \frac{\rho K}{2}\Big)=\Pb_{\bar x }(\Upsilon \leq T)\\
  &\leq \Pb_{\bar x }\Big(\Upsilon \leq T,\sum_{i=0}^{KN} |X^i_T-X^i_0|\geq \frac{\rho K}{4}\Big) +\Pb_{\bar x }\Big(\Upsilon \leq T,  \sum_{i=1}^{KN}|X^i_T-X^i_{\Upsilon}|\geq \frac{\rho K}{4}\Big)\\
  &\leq \Pb_{\bar x }\Big(\sum_{i=0}^{KN} |X^i_T-X^i_0|\geq \frac{\rho K}{4}\Big) + \E_{\bar x} \Big[ \int_0^T \Pb_{\bar X_s }\Big(\sum_{i=1}^{KN} |X^i_{T-s}-X^i_{0}|\geq \frac{\rho K}{4}\Big) \one_{\Upsilon\in ds}\Big]\\
  & \leq 2 \underset{\bar y \in \CX^{KN}, s\in [0,T]}{\sup} \Pb_{\bar y }\Big(\sum_{i=0}^{KN} |X^i_s-X^i_0|\geq \frac{\rho K}{4}\Big).
 \end{aligned}
\end{equation}
The aim is thus to find an upper bound on the last term for any $\bar x\in \CX^{KN}$ and any $s\in [0,T]$. Using Markov's inequality,
\begin{equation}
\label{eq_markovineq}
 \Pb_{\bar x }\Big(\sum_{i=0}^{KN} |X^i_s-X^i_0|\geq \frac{\rho K}{4}\Big)\leq e^{-\frac{K\rho}{4\sqrt{T}}} \prod_{i=1}^{KN} \E_{x^i}\Big[ e^{\frac{|X^i_s-x^i|}{\sqrt{s}}}\Big].
\end{equation}
If we denote the kernel of the semigroup $P^u_s$ of the reflected diffusion process by $p^u_s(x,y)$, Part 3 in~\cite{wang_yan_2013} and the fact that $\bar \CX$ is compact imply that there exist two positive constants $C_1, C_2$ such that for any $x,y\in \CX$,
\begin{equation}
\label{eq_kernel}
 p^u_s(x,y) \leq \frac{C_1}{s^{d/2}} e^{-\frac{|x-y|^2}{C_2 s}}.
\end{equation}
Thus, using~\eqref{eq_kernel} then a change of variables, we find that there exists $C_3>0$ independent from $s$ such that 
\begin{equation}
\label{eq_boundEexp}
\begin{aligned}
 \E_{x^i}\Big[ e^{\frac{|X^i_s-x^i|}{\sqrt{s}}}\Big] \leq \int_{\R^d} \frac{C_1}{s^{d/2}} e^{-\frac{|x-y|^2}{C_2 s}}e^{\frac{|y-x|}{\sqrt{s}}}dy
  \leq \int_{\R^d} C_1 e^{-\frac{|z|^2}{C_2 }}e^{|z|}dz =C_3 <+\infty.
 \end{aligned}
\end{equation}
We deduce with the last line in~\eqref{eq_sup}, \eqref{eq_markovineq} and \eqref{eq_boundEexp} that
\begin{equation*}
 \Pb_{\bar x }\Big(\underset{t\in[0,T]}{\sup} \sum_{i=0}^{KN} |X^i_t-X^i_0|\geq \frac{\rho K}{2}\Big) \leq 2e^{-K\left(\frac{\rho}{4\sqrt{T}}-\ln(C_3) N\right)},
\end{equation*}
where $\frac{\rho}{4\sqrt{T}}-\ln(C_3) N$ tends to $+\infty$ when $T$ tends to $0$, i.e. there exists $T_1$ such that for all $T\leq T_1$, $\frac{\rho}{4\sqrt{T}}-\ln(C_4) N\geq C$.\\
\noindent
Finally, concerning the third term of \eqref{eq_3termesbis}, let $r=\log \left(\frac{\rho}{2(\bar{b}+\bar{d}+\bar{c}N)NT}\right)$,
\begin{eqnarray*}
\Pb\Big( \underset{t\in[0,T]}{\sup} \frac{\mathcal{P}(t)}{K} \geq \frac{\rho}{2}\Big) & \leq & \Pb\left( e^{r \mathcal{P}(T)} \geq e^{r\frac{K\rho}{2}}\right)\\
& \leq & e^{-\frac{\rho r K}{2}}\E [e^{r\mathcal{P}(T)}]\\
%&\leq & \exp\left(-\frac{\rho r K}{2}+(\bar{b}+\bar{d}+\bar{c}N)KNT(e^r-1)\right)\\
& \leq & e^{ \left( - K \left( \frac{\rho r}{2}-(\bar{b}+\bar{d}+\bar{c}N)NT(e^r-1) \right)\right)},\\
& \leq& e^{ \left( -K \left( \frac{\rho}{2} \left[\log \left(\frac{\rho}{2(\bar{b}+\bar{d}+\bar{c}N)NT}\right)-1\right]+(\bar{b}+\bar{d}+\bar{c}N)NT \right) \right)}.
\end{eqnarray*}
There exists $T_2\leq T_1$ such that for all $T\leq T_2$, $\frac{\rho}{2} \left[\log \left(\frac{\rho}{2(\bar{b}+\bar{d}+\bar{c}N)NT}\right)-1\right]+(\bar{b}+\bar{d}+\bar{c}N)NT\geq C$. 
We conclude the proof~: for all $T \leq T_2$,
\begin{equation*}
\underset{K\to+\infty}{\limsup}\; \frac{1}{K} \log \Pb(\underset{t\in[0,T]}{\sup} \W(\nu^{K,u}_T,\nu^{K,u}_0) \geq  \rho)\leq -C.
\end{equation*}
\end{proof}

%%%%%%%%%%%%%%%%%%%%%%%%%%%%%%%%%%%%%%%%%%%%%%%%%%%%%%%%%%%%%%%%%%%%
%%%%%%%%%%%%%%%%%%%%%%%%%%%%%%%%%%%%%%%%%%%%%%%%%%%%%%%%%%%%%%%%%%%%%%
\section{Survival probability for a branching diffusion process}
\label{sec_survivalproba}

In this part, we study a distinct model which is a branching diffusion process. This model is correlated with the other one as explained at the end of Section~\ref{sec_theo}.
Any individual is characterized by its location $X^i_t\in \CX$ which is described as before by a diffusion process normally reflected at the boundary of $ \CX$, \eqref{eq_brownien}, with the diffusion coefficient $m>0$. Moreover, each individual with location $x\in \CX$ gives birth to a new individual at rate $b(x)$ and dies at rate $d(x)$. Those rates are assumed to have some regularity that we detail.
\begin{assumption}
\label{hyp_bd}
 $b,d$ are two Lipschitz functions, $b$ is a positive function and there exists $\underbar{b}, \bar{b}, \bar{d}$ such that for all $x\in\CX$, $\underbar b <b(x)\leq \bar{b}$ and $0 \leq d(x)\leq \bar{d}$. Moreover, $d$ is a non-zero function.
\end{assumption}
\noindent
Let $M_t$ denote the number of individuals at time $t$. We describe the dynamics of the diffusion process at each time by the finite measure 
$$
 \eta_t=\sum_{i=1}^{M_t}\delta_{X^i_t}.
$$
The aim of this part is to describe the survival probability of the population using its parameters. Let
\begin{equation*}
  \Upsilon_0=\inf \{ t\geq 0, M_t=0\}.
\end{equation*}

\noindent
The first Theorem concerns the survival probability of the population assuming that, initially, there is only one individual at location $x$. This probability is characterized as a solution to an elliptic differential equation on the location space $\CX$. Remark that the location of the first individual plays a main role. Indeed, if the first individual appears in a place where the growth is low or negative, it has a high probability to die with no descendants.\\
We denote the probability measure under which $\eta_0=\delta_x$ by $\Pb_{\de_x}$.
\begin{theorem}
\label{theo_survieinf}
Let $H$ be the principal eigenvalue of the elliptic operator $m\Delta_x .+(b-d).$ with Neumann boundary conditions on $\CX$, see \eqref{eq_defH}. 
If $H>0$, there exists a unique positive $C^2$-solution $\phi^*$ to the elliptic equation
\begin{equation}
\label{eq_wstar}
 \left\{
\begin{aligned}
 &0=m\Delta_x \phi^*(x)+(b(x)-d(x))\phi^*(x)-b(x)\phi^*(x)^2, \; \forall x\in \CX,\\
 &\p_n \phi^*(x)=0, \; \forall x\in \p\CX,\\
\end{aligned}
\right.
\end{equation}
and $\phi^*(x)=\lim_{t\to \infty} \Pb_{\delta_x}[\Upsilon_0\geq t]$ for all $x\in \CX$.\\
If $H\leq 0$, \eqref{eq_wstar} has no non negative solution, we set $\phi^*\equiv 0$, and $\lim_{t\to \infty} \Pb_{\delta_x}[\Upsilon_0\geq t]=0$ for all $x\in \CX$.
\end{theorem}

\indent
The second result of this part estimates the probability that the population size is of order $K$ after a logarithmic time $\log K$.
For all $\epsilon>0$ and $K \in \N^*$, we set
\begin{equation*}
 \Upsilon_{\epsilon K}=\inf\{ t \geq 0, \langle \eta_t, \one \rangle \geq \epsilon K \}.
\end{equation*}

\begin{theorem}
\label{theo_explosion}
 Let $\epsilon>0$ and $(t_K)_{K>0}$ be a sequence of times such that
$
 \lim_{K\to +\infty} t_K/\log( K)=+\infty.
$
Then for all $x\in \CX$,
\begin{equation*}
 \lim_{K \to +\infty} \Pb_{\delta_x}[\Upsilon_{\epsilon K} <t_K]=\phi^*(x).
\end{equation*}
\end{theorem}

The end of this part is devoted to the two proofs.

\begin{proof}[Proof of Theorem \ref{theo_survieinf}]

We first study the behavior of the probability $\Pb_{\delta_x}[\Upsilon_0\leq t]$. We denote the time of the first event (birth or death) of the population by $E_1$. The law of $E_1$ is given by $\Pb_{\de_x}[E_1\leq t]=\E_{\de_x}\bigg[ \int_0^t (d(X_s)+ b(X_s)) e^{-\int_0^s (b(X_r)+d(X_r))dr} ds \bigg]$. We set $I(s):=\int_0^s (b(X_r)+d(X_r))dr$. Using the Markov property of $\eta$, we obtain
\begin{eqnarray*}
 \Pb_{\de_x}[\Upsilon_0\leq t] &=& \E_{\de_x}\Big[ \one_{E_1\leq t} \one_{\{M_{E_1}=0\}} +\one_{E_1 \leq t } \one_{\{M_{E_1}=2\}} \E_{2\de_{X_{E_1}}}[\one_{M_{t-E_1}=0}] \Big]\\
 &=& \E_x \bigg[ \int_0^t \Big(d(X_s)+ b(X_s)\Pb_{\de_{X_{s}}}[\Upsilon_0\leq t-s]^2\Big) e^{-I(s)} ds \bigg],
\end{eqnarray*}
where $X$ under $\Pb_x$ is solution to \eqref{eq_brownien} with the initial condition $x$ and the diffusion coefficient $m$.
Thus $g(x,t)=\Pb_{\de_x}[\Upsilon_0\leq t]$ satisfies for all $x\in \CX$, and all $t>0$,
\begin{equation}
\label{eq_feynmankac}
\left\{
\begin{aligned}
 &g(x,t)=\E_x \bigg[ \int_0^t \Big( d(X_s)+ b(X_s)g(X_s,t-s)^2 \Big) e^{-I(s)} ds \bigg],\;\forall (x,t) \in \CX\times \R^+,\\
&g(x,0)=0,\; \forall x\in \CX.
\end{aligned}
\right.
\end{equation}
 Using Gronwall's Lemma for bounded functions, we deduce immediately that \eqref{eq_feynmankac} has a unique bounded solution.\\
\indent We will now show that there exists a unique $C^2$-solution to 
\begin{equation}
\label{eq_v}
 \left\{
\begin{aligned}
 &\p_t f(x,t)=m\Delta_x f(x,t)-\big(b(x)+d(x)\big)f(x,t)+d(x)+b(x)f(x,t)^2, \; \forall (x,t)\in \CX\times \R^+,\\
 &\p_n f(x,t)=0, \; \forall (x,t)\in \p\CX\times \R^+\\
 &f(x,0)=0, \; \forall x\in \CX,
\end{aligned}
\right.
\end{equation}
such that $f\in C^{2,1}(\CX \times \R^+)$, is positive and is smaller than $1$ by using super- and sub-solutions arguments. Indeed, let $F(x,f)=-\big(b(x)+d(x)\big)f+d(x)+b(x) f^2$.
We easily see that $\underline{f}\equiv 0$ and $\bar{f}\equiv1$ satisfy:
\begin{eqnarray*}
& \left\{
\begin{aligned}
 &\p_t \underline{f}\leq m\Delta_x \underline{f}+F(x,\underline{f}), \; \forall (x,t)\in \CX \times \R^+,\\
&\p_t \bar{f}\geq m\Delta_x \bar{f}+F(x,\bar{f}), \; \forall (x,t)\in \CX \times \R^+,
\end{aligned}
\right. \\
& \underline{f}(x,0)\leq f(x,0) \leq \bar{f}(x,0), x\in \CX,\\
& \p_n\underline{f}(x,t)\leq 0 \leq \p_n\bar{f}(x,0), x\in \p\CX, t \in \R^+.
\end{eqnarray*}
That is, $\underline{f}$ (resp. $\bar{f}$) is a sub-solution (resp. super-solution) to \eqref{eq_v}. Moreover, $F$ and $\p_fF$ belong to $ C(\CX\times \R)$ and $F$ is a Lipschitz function with respect to $x$ by means of Assumptions~\ref{hyp_bd}. We apply Theorem 4 of Chapter III of \cite{roques_2013} to deduce that \eqref{eq_v} admits a solution $f\in C^{2,1}(\CX\times \R^+)$ satisfying $0 \leq f \leq 1$. The uniqueness of the solution is a consequence of the maximum principle.\\
\indent The next step is to use a Feynman-Kac formula to deduce that $f$ is also a solution to \eqref{eq_feynmankac}.
Indeed, for $X$ solution to \eqref{eq_brownien}, let, for all $t\geq 0$, and $s\in [0,t]$,
 \begin{equation*}
  H(s,X_s)=f(X_s,t-s) e^{-I(s)}.
 \end{equation*}
Applying Itô Formula to $H(s, X_s)$, and using \eqref{eq_brownien}, \eqref{eq_v} and the fact that $\p_n f(x,t)=0$ for all $x\in \p \CX$, we find for all $s\in [0,t]$,
\begin{equation}
\label{eq_ito}
\begin{aligned}
 H(s,X_s)= H(0,X_0)&-\int_0^s \Big(d(X_{\sigma})+ b(X_{\sigma})f(X_{\sigma},t-{\sigma})^2\Big) e^{-I(\sigma)}d{\sigma}\\
 &+\int_0^s \sqrt{2m}(\p_x f(X_{\sigma},t-{\sigma})) e^{-I(\sigma)}dB_{\sigma}.
 \end{aligned}
 \end{equation}
We take the expectation of \eqref{eq_ito} for $s<t$. The expectation of the last term is equal to $0$ as $(\int_0^r \p_x f(X_{\sigma},t-{\sigma}) e^{-I(\sigma)}dB_{\sigma})_{r\in[0,s]}$ is a martingale. In addition, $\E_x[H(0,X_0)]=f(x,t)$. It stays to make $s$ tend to $t$ using the dominate convergence Theorem. As $\E_x[H(s,X_s)]\underset{s\to t}{\to}\E_x[H(t,X_t)]=0$, we deduce that $f$ is a solution to \eqref{eq_feynmankac}. \eqref{eq_feynmankac} admits a unique bounded solution, thus both solutions are equal, i.e. $\Pb_{\delta_x}[\Upsilon_0\leq t]=f(x,t)$. \\
\indent
Finally, we deduce the survival probability $\lim_{t\to+\infty}(1-f(x,t))$ using results on Equation \eqref{eq_v} described in \cite{berestycki_hamel_roques_2005} and in Theorems 9 and 11 of Chapter III in \cite{roques_2013}. Indeed, they prove that if $H>0$, there exists a unique positive solution $\phi^*$ to the elliptic equation \eqref{eq_wstar}, and that $\phi(t,x)=1-f(x,t)=\Pb_x(\Upsilon_0>t)$ tends to $ \phi^*(x)$ in $C^2(\CX)$ as $t\to +\infty$. Moreover, if $H\leq 0$, the unique solution to \eqref{eq_wstar} is the zero function and $\phi(t,x)\to 0$ uniformly in $\CX$ as $t\to +\infty$.
\end{proof}

%%%%%%%%%%%%%%%%%%%%%%%%%%%%%%%%%%%%%%%%%%%%%%%%%%%%%%%%%%%%%%%
%%%%%%%%%%%%%%%%%%%%%%%%%%%%%%%%%%%%%%%%%%%%%%%%%%%%%%%%%%%%%%%%%

\begin{proof}[Proof of Theorem \ref{theo_explosion}]
First, we split the studied probability into three parts~:
\begin{equation}
\label{eq_3termes}
\begin{aligned}
 \Pbx(\Upsilon_{\epsilon K}& < t_K)=\Pbx(\Upsilon_{\epsilon K} < t_K, \log\log(K)<\Upsilon_0<+\infty)\\
 &+\Pbx(\Upsilon_{\epsilon K} < t_K, \log\log(K)\geq \Upsilon_0)
+\Pbx( \Upsilon_{\epsilon K} < t_K, \Upsilon_0=+\infty).
\end{aligned}
\end{equation}
Let us start with the first term of \eqref{eq_3termes}~:
\begin{equation*}
 \Pbx(\Upsilon_{\epsilon K} < t_K, \log\log(K)<\Upsilon_0<+\infty)\leq \Pbx( \log\log(K)<\Upsilon_0<+\infty)\underset{K \to +\infty}{\to} 0.
\end{equation*}
The second term of \eqref{eq_3termes} will be treated using a comparison with a pure birth process. Let us consider a birth process with constant birth rate $\bar{b}$ and started with only one individual. $\tilde{\Upsilon}_{\epsilon K}$ denote the first time when $\tilde{N}$, the population size of the process, is greater than $\epsilon K$.
\begin{align*}
 \Pbx(\Upsilon_{\epsilon K} < t_K, \log\log(K)\geq \Upsilon_0) &\leq \Pbx(\Upsilon_{\epsilon K} \leq \log\log (K))\\
& \leq \Pb_1(\tilde{\Upsilon}_{\epsilon K} \leq \log \log(K))\\
& \leq \Pb_1 (\tilde{N}_{\log\log K} \geq \epsilon K)\\
%& \leq \E_1[\tilde{N}_{\log\log K}]/ (\epsilon K)\\
& \leq e^{\bar{b} \log \log K}/(\epsilon K) \underset{K\to +\infty}{\to}0.
\end{align*}
There remains to deal with the third term in~\eqref{eq_3termes}. Note that if $H\leq 0$, Theorem~\eqref{theo_survieinf} implies that the third term is equal to zero, and the proof is done.\\
From this point forward, we assume that $H>0$.
Let $h$ be a positive eigenvector of the operator $m\Delta_x.+(b-d).$ with Neumann boundary conditions on $\p \CX$ associated with the eigenvalue $H$. Thanks to Itô's Formula, we find
\begin{equation*}
 \langle \eta_t, e^{-Ht}h\rangle =\langle \eta_0, h\rangle +\int_0^t \langle \eta_s, e^{-Hs}m\Delta_xh + (b-d)he^{-Hs}-Hhe^{-Hs} \rangle ds+ \int_0^t \sum_{i=1}^{M_s} \sqrt{2m} \nabla_x h(X^i_s)e^{-Hs}dB^i_s,
\end{equation*}
As $m\Delta_x h+(b-d)h=Hh$ and $\nabla_xh$ is bounded on $\CX$, $(e^{-Ht}\langle \eta_t, h \rangle)_{t\geq 0}$ is a martingale. Moreover it is positive, so, it converges a.s. to a non-negative random variable that will be denoted by $W$.
Obviously, $\{ \Upsilon_0 <+\infty\} \subset \{ W=0\}$. Our aim is to prove that this is an a.s. equality. As done in the previous proof, we denote the time of the first event of the population by $E_1$ and we set $I(s):=\int_0^s (b(X_r)+d(X_r))dr$.
Using the Markov property and the independence between individuals, we find an equation satisfied by $\Pbx[W=0]$:
\begin{equation}
\label{eq_vbar}
 \Pbx[W=0]= \E_x\left[ \int_0^{+\infty} e^{-I(s)} \Big(d(X_s)+b(X_s)\Pb_{\delta_{X_{s}}}[W=0]^2 \Big)ds \right].
\end{equation}
\noindent Finally, $g(x):=\Pbx[W>0]=1-\Pbx[W=0]$ is solution to
\begin{equation}
\label{eq_wsto}
 g(x)=\E_x\left[\int_0^{+\infty} e^{-I(s)} b(X_s)g(X_s)(2-g(X_s))ds\right], \forall x\in \bar\CX.
\end{equation}
Let us show that there exists at most one non-zero solution with value in $[0,1]$ to Equation \eqref{eq_wsto}.
Let $g_1$ and $g_2$ be two such solutions. We define
\begin{equation*}
 \gamma =\sup\{\tilde{\gamma}>0, g_1(x)-\tilde{\gamma} g_2(x)\geq 0, \forall x \in \bar\CX \}.
\end{equation*}
Assume first that $\gamma<1$.
\begin{equation}
\label{eq_w1w2}
 g_1(x)-\gamma g_2(x)=\E_x\bigg[ \int_0^{+\infty} e^{-I(s)}b(X_s)\Big[2(g_1(X_s)-\gamma g_2(X_s))
-(g_1(X_s)^2-\gamma g_2(X_s)^2)\Big]ds \bigg].
\end{equation}
As $\gamma<1$, $g_1^2-\gamma g_2^2\leq (g_1-\gamma g_2)(g_1+\gamma g_2)$ and we find
\begin{equation*}
 g_1(x)-\gamma g_2(x) \geq \E_x \bigg[ \int_0^{+\infty} e^{-I(s)}b(X_s)\big[g_1(X_s)-\gamma g_2(X_s)\big] \big[2-g_1(X_s)-\gamma g_2(X_s)\big]ds \bigg].
\end{equation*}
Moreover, by the definition of $\gamma$, there exists $x_0\in \bar\CX$ such that $g_1(x_0)-\gamma g_2(x_0)=0$, so
\begin{equation*}
0= \E_{x_0} \bigg[ \int_0^{+\infty} e^{-I(s)}b(X_s)\big[g_1(X_s)-\gamma g_2(X_s)\big] \big[2-g_1(X_s)-\gamma g_2(X_s)\big]ds \bigg].
\end{equation*}
Thus for a.e. $s\in \R^+$, $\Pb_{x_0}$-a.s., $b(X_s)[g_1(X_s)-\gamma g_2(X_s)][2-g_1(X_s)-\gamma g_2(X_s)]=0$.\\
Let us note that for all $x\in \bar\CX$, for $i=1,2$, \eqref{eq_vbar} implies that
\begin{eqnarray*}
 1-g_i(x)\geq \E_x\left[ \int_0^{+\infty} e^{-I(s)}d(X_s)ds \right]
 =  \Pbx[M_{E_1}=0]>0,
\end{eqnarray*}
that is, for all $x\in \bar\CX$, $2-g_1(x)-\gamma g_2(x)>0$. As $b$ is positive, for a.e. $s\in \R^+$, $\Pb_{x_0}$-a.s.
$
g_1(X_s)-\gamma g_2(X_s)=0.
$
In addition with the fact that $\gamma g_2^2-g_1^2=\gamma g_2^2 (1-\gamma)-(g_1-\gamma g_2)(g_1+\gamma g_2)$, \eqref{eq_w1w2} implies that,
\begin{eqnarray*}
0&=& g_1(x_0)-\gamma g_2(x_0)\\
&=&\E_{x_0}\bigg[ \int_0^{+\infty} e^{-I(s)}b(X_s)\gamma g_2(X_s) [1- \gamma ]ds \bigg].
\end{eqnarray*}
Using the same argument as before and that $1- \gamma>0$, we deduce that for a.e. $s\in \R^+$, $\Pb_{x_0}$-a.s.,
$
g_2(X_s)=0.
$
Moreover, under $\Pb_{x_0}$, the random variable $X_s$ has a density with respect to Lebesgue measure, that is, for Lebesgue-a.a. $x\in \bar\CX$, $g_2(x)=0$. 
This is a contradiction with the fact that $g_2$ is a non-zero solution.\\
Finally if $\gamma \geq 1$, we define instead $\gamma' =\sup\{\tilde{\gamma}>0, g_2(x)-\tilde{\gamma} g_1(x)\geq 0, \forall x \in \CX \}<1$ and we use symmetric arguments to reach a contradiction. Thus, there is at most one solution to \eqref{eq_wsto} with values in $[0,1]$.\\

\noindent
The next step is to show that $\phi^*$, which is solution to \eqref{eq_wstar}, is also solution to \eqref{eq_wsto}.
Let us write $f^*=1-\phi^*$, which satisfies
\begin{equation}
\label{eq_ustar}
 \left\{ \begin{aligned}
          &0=m\Delta_x f^* -(b+d)f^*+d+b(f^*)^2 , \text{ on } \CX,\\
	  &\p_nf^*=0, \text{ on } \p\CX.
         \end{aligned}
\right.
\end{equation}
We apply Itô's Formula to $f^*(X_t)e^{-\int_0^t(b(X_r)+d(X_r))dr}$. Then taking the expectation and using Equation \eqref{eq_ustar}, we deduce
\begin{align*}
 f^*(x)=&\E_x\left[\int_0^t  e^{-I(s)}[d(X_s)+b(X_s)f^*(X_s)^2]ds\right]
+\E_x\left[f^*(X_t)e^{-I(t)}\right].
\end{align*}
Our aim is now to let $t$ tend to infinity.
Note that $I(t)\geq \underbar b t$ for all $t\in \R^+$ and that $f^*$ is bounded by $1$. Hence, we use the dominated convergence Theorem to find those two convergences:
\begin{align*}
 &\left\vert \E_x \left[f^*(X_t)e^{-I(t)}\right]\right\vert \leq \E_x\left[e^{-I(t)}\right]\underset{t\to +\infty}{\to} 0\\
 &\left\vert \E_x\left[\int_t^{+\infty} e^{-I(s)}[d+b(f^*)^2](X_s)ds \right]\right\vert \leq \E_x\left[\int_t^{+\infty} e^{-I(s)}(\bar{b}+\bar{d})ds \right]\underset{t\to +\infty}{\to} 0.
\end{align*}
Thus, we make $t$ tend to infinity and we find, for all $x\in \bar\CX$,
\begin{equation*}
 f^*(x)=\E_x\left[\int_0^{+\infty}  e^{-I(s)}[d(X_s)+b(X_s)f^*(X_s)^2]ds\right].
\end{equation*}
Since $\phi^* \equiv 1-f^*$, we conclude that
$\phi^*$ is a solution to \eqref{eq_wsto}. There exists at most one non-zero solution to~\eqref{eq_wsto}, thus we have either $\Pbx[W>0]=\phi^*(x)$ for all $x\in \bar\CX $, or $\Pbx[W>0]=0$ for all $x\in \bar\CX $.
Using Itô's Formula, it is easy to check that in the case $H>0$, $(\langle \eta_t,e^{-Ht}h \rangle)_{t\geq 0}$ is bounded in $L^2$. So this martingale is uniformly bounded and it converges in $L^1$ to $W$, hence $\E_{\delta_x}[W]=h(x)>0$. Finally,
\begin{equation}
\label{eq_Wpositive}
\Pbx[W>0]=\phi^*(x)=\Pbx[\Upsilon_0=+\infty] \; \Rightarrow \; \{ W>0\} =\{\Upsilon_0=+\infty\} \; a.s.
\end{equation}
\noindent 
Thus, on $\{\Upsilon_0=+\infty\}$,
\begin{align*}
 \dfrac{\log(\epsilon K)}{\Upsilon_{\epsilon K}}\geq  \dfrac{\log(\langle \eta_{\Upsilon_{\epsilon K} } ,  h e^{-H\Upsilon_{\epsilon K}} \rangle .\|h\|_{\infty}^{-1} e^{H\Upsilon_{\epsilon K}})}{\Upsilon_{\epsilon K}}\underset{K \to +\infty}{\to} H>0 \; a.s.,
\end{align*}
as $\langle \eta_{\Upsilon_{\epsilon K} } ,  h \rangle \leq \epsilon K \|h\|_{\infty}  $, $\Upsilon_{\epsilon K} \to +\infty$ when $K$ tends to infinity and $W>0$.
Hence,
\begin{equation*}
 \lim_{K \to +\infty} \dfrac{\Upsilon_{\epsilon K}}{\log(\epsilon K)} < +\infty \text{ and } \lim_{K \to +\infty} \dfrac{t_K}{\log(\epsilon K)}=+\infty,
\end{equation*}
and so, the third term in~\eqref{eq_3termes} satisfies
\begin{equation}
\label{eq_limittK}
 \Pbx( \Upsilon_{\epsilon K} < t_K, \Upsilon_0=+\infty)=\Pbx\left( \frac{\Upsilon_{\epsilon K}}{\log(\epsilon K)} < \frac{t_K}{\log(\epsilon K)}, \Upsilon_0=+\infty\right) \underset{K\to\infty}{\rightarrow}\Pbx\left(  \Upsilon_0=+\infty\right).
\end{equation}
Finally, we have shown that the two first terms in \eqref{eq_3termes} tend to $0$, so using additionally \eqref{eq_limittK}
\begin{equation*}
 \lim_{K\to +\infty } \Pbx( \Upsilon_{\epsilon K} < t_K)=\Pbx\left(  \Upsilon_0=+\infty\right)=\phi^*(x).
\end{equation*}
That ends the proof for $H>0$.
\end{proof}

\section{Proof of Theorem~\ref{theo_main}}
\label{sec_TSS}

This section is devoted to prove Theorem~\ref{theo_main}.
The structure of the proof of Theorem~\ref{theo_main} is similar to the one of \cite{champagnat_2006}, thus, we do not repeat all the details but only the points that are different.\\
The first proposition concerns the behavior of the first time of a mutation $S_1^K$, when the initial state is a monomorphic population.
\begin{proposition}
\label{prop_mono}
 Suppose that Assumptions~\ref{ass_coef} and~\eqref{ass_KqK} hold. Let $u\in \CU$ and $\mathcal{C}^u$ a compact subset of $M_F(\CX \times \{u\})$ such that $0\not\in \mathcal{C}^u$. If $\nu^K_0 \in  \mathcal{C}^u \cup M_F^K(\CX)$, then
 \begin{itemize}
  \item for any $\gamma>0$, $\underset{K\to +\infty}{\lim} \Pb_{\nu^K_0}\left( S_1^K>\log K , \, \underset{t\in[\log K, S_1^K]}{\sup} \W(\nu^K_t,\bar \xi^u\delta_u) \geq \gamma \right)=0$.\\
 \item Furthermore, $\underset{K \to +\infty}{\lim} \Pb_{\nu^K_0}\left( S_1^K > \frac{t}{K q_K} \right)=\exp\Big( -t \int_{\CX}b^u(x)p^u(x)\bar \xi^u(dx) \Big).$
 \end{itemize}

\end{proposition}

 Proposition~\ref{prop_mono} can be proved using similar arguments as those of the proof of Lemma 2 in \cite{champagnat_2006}. It is a consequence of the following lemma.
 \begin{lemma}
 \label{lemma_enterinball}
  For any $\alpha>0$, there exists $T_\alpha>0$ such that for any $\xi_0\in \mathcal{C}^u$, for any $t\geq T_\alpha$, $\W(\xi_t,\bar \xi^u)<\alpha$, where $(\xi_t)_{t\geq 0}$ is the solution to Equation \eqref{eq_theoapproxpop} with initial state $\xi_0$.\\
 \end{lemma}
 \begin{proof}
On the one hand, Theorem 1.4 in \cite{leman_mirrahimi_meleard_2014} implies that the density of $\bar\xi^u$ is a stable monomorphic equilibrium for the $L^2$-distance. Using a proof in three steps as that of Proposition~\ref{prop_stability}, we prove a $\W$-stability~: there exists $\alpha'$ such that for any $\xi_0\in B(\bar\xi^u,\alpha')$ and for any $t\geq 0$, $\W(\xi_t,\bar\xi^u)< \alpha$.\\
On the other hand, for any $\xi_0\in \mathcal{C}^u$, $\xi_t$ converges towards $\bar\xi^u$. There exists $T_{\xi_0}$ such that $\W(\xi_{T_{\xi_0}},\bar\xi^u)<\alpha'/2$.
Using Lemma~\ref{lemma_lipschitz} and arguments similar to \eqref{eq_majeloignement} and \eqref{eq_debut}, we show that for any $t \geq 0$, any $\zeta_0 \in M_F(\CX)$,
$$
\sup_{r\in[0,t]}\W(\zeta_{t},\xi_{t})\leq C(t) \W(\zeta_0,\xi_0),
$$
where $(\zeta_t)_{t\geq 0}$ is the solution to \eqref{eq_theoapproxpop} with initial state $\zeta_0\in M_F(\CX\times \{u\})$.
Consequently, there exists $\delta_{\xi_0}>0$ such that for any $\zeta_0 \in B(\xi_0,\delta_{\xi_0})$, $\W(\zeta_{T_{\xi_0}},\xi_{T_{\xi_0}}) < \alpha'/2$.
Thus, for any $\zeta_0$ such that $\W(\zeta_0,\xi_0)< \delta_{\xi_0}$, for any $t\geq T_{\xi_0}$, $\W(\zeta_t,\bar\xi^u)< \alpha$.\\
Finally, as $\mathcal{C}^u$ is a compact set, there exists a finite number of balls such that $\mathcal{C}^u \subset \cup_{i=1}^n B(\xi^i_0,\delta_{\xi^i_0})$. Defining $T_\alpha=\max_{i=1..n}T_{\xi^i_0}$, we deduce the lemma.
 \end{proof}

 \begin{proof}[Proof of Proposition~\ref{prop_mono}]
First, remark that the first probability of Proposition~\ref{prop_mono} is non-increasing with $\gamma$. Thus, it is sufficient to prove the property for any small $\gamma>0$. Let us assume that $\gamma$ satisfies the assumptions of Theorem~\ref{theo_exittimeintro}. Theorem~\ref{theo_exittimeintro} in the monomorphic case implies that :
there exist $\gamma'>0$, $V>0$, such that 
\begin{equation}
\label{eq_term1}
\underset{\nu^K_0\in B(\bar\xi^u,\gamma') \cap M_F^K}{\sup} \Pb\left(R^K_\gamma \leq S_1^K \wedge e^{KV} \right) \underset{K\to \infty}{\to} 0,
\end{equation}
$R^K_\gamma$ is defined by~\eqref{def_RKgamma}. We set $2\alpha=\gamma'$, then Lemma~\ref{lemma_enterinball} and Theorem~\ref{theo_majoration} imply that
\begin{equation}
\label{eq_term2}
\underset{\nu^K_0\in \mathcal{C}^u}{\sup} \Pb\left( \W(\nu^K_{T_\alpha},\bar\xi^u)\geq 2\alpha \right) \underset{K\to+\infty}{\to}0.
\end{equation}
Using the Markov property, we deduce if $K$ is sufficiently large such that $\log(K)>T_{\alpha}$,
\begin{align*}
&\Pb_{\nu^K_0}\bigg( S_1^K>\log K , \, \underset{t\in[\log K, S_1^K]}{\sup} \W(\nu^K_t,\bar \xi^u) \geq \gamma \bigg)\\
&\leq  \Pb_{\nu^K_0}(\W(\nu^K_{T_\alpha},\bar\xi^u)\geq \gamma')+\E_{\nu^K_0}\left[\one_{\{ S_1^K\geq \log(K) , \W(\nu^K_{T_\alpha},\bar\xi^u)<\gamma'\}} \underset{t\in[T_\alpha, S_1^K+T_\alpha]}{\sup} \W(\nu^K_t,\bar \xi^u) \geq \gamma\right]\\
& \leq \Pb_{\nu^K_0}(\W(\nu^K_{T_\alpha},\bar\xi^u)\geq \gamma') + \E_{\nu^K_0}\left[ \one_{\{S_1^K\geq \log(K), \W(\nu^K_{T_\alpha},\bar\xi^u)<\gamma'\}} \Pb_{\nu^K_{T_\alpha}}(R^K_\gamma \leq S_1^K )\right]\\
& \leq \Pb_{\nu^K_0}(\W(\nu^K_{T_\alpha},\bar\xi^u)\geq 2\alpha) + \underset{\nu_0\in B(\bar\xi^u,\gamma')}{\sup} \Big[\Pb_{\nu_0}(R^K_\gamma \leq S_1^K \wedge e^{KV} ) +  \Pb_{\nu_0}(e^{KV} \leq R^K_\gamma \leq S_1^K ) \Big].
\end{align*}
The first term and the second term tend to $0$ when $K$ tends to $+\infty$ according to \eqref{eq_term2} and \eqref{eq_term1} respectively. There remains the third term. On $\{t\leq R^K_\gamma\leq S_1^K\}$, the number of mutations $\mathcal{M}_t$ is stochastically bounded from below by a Poisson process with parameter $K q_K (\langle b^u p,\bar\xi^u\rangle - \gamma \|b^u p\|_{Lip})$ which is positive if $\gamma$ is small enough. We conclude with the fact that $\Pb(e^{KV} \leq R^K_\gamma \leq S_1^K)\leq \Pb(\mathcal{M}_{e^{KV}}=0) \underset{K \to +\infty}{\to} 0$, under Assumption~\eqref{ass_KqK}. That ends the proof of the first point.\\
The second point of Proposition~\ref{prop_mono} is easily deduced from this first point and Lemma 2 in \cite{champagnat_2006}.
\end{proof}

The second proposition studies the process with a dimorphic initial state. Let us define~:
\begin{itemize}[noitemsep]
 \item $\theta_0$ is the first time when the population becomes monomorphic again,
 \item $V_0$ is the phenotypic trait of the population at time ${\theta_0}$.
\end{itemize}

\begin{proposition}
\label{prop_dimo}
 Suppose Assumptions~\ref{ass_coef}, \ref{ass_noncoexistence} and \eqref{ass_KqK} hold, and that the initial state $\nu^K_0\in M_F(\CX\times \{u,v\})$ is such that $\nu^{K,u}_0 $ converges weakly to $\bar \xi^u$ in $M_F(\CX)$ and $\nu^{K,v}_0=\frac{\delta_{x_0}}{K}$. Then,
 $$
 \underset{K\to \infty}{\lim } \Pb(V_0=v)=1-\underset{K\to \infty}{\lim } \Pb(V_0=u)=\phi^{vu}(x_0).
 $$
Moreover,
 $$
 \forall \eta>0, \, \underset{K\to\infty}{\lim}\Pb\left(\theta_0\leq S_1^K \wedge \frac{\eta}{Kq_K}\right)=1,\\
 \text{ and }\forall \gamma>0, \, \underset{K\to\infty}{\lim}\Pb\left(\W(\nu^K_{\theta_0},\bar\xi^{V_0}<\gamma\right)=1.
 $$
\end{proposition}

\begin{proof}
 We set $\gamma>0$ small enough to use Theorem~\ref{theo_exittimeintro}~: there exist $\gamma',\epsilon, V$ such that 
 $$
 \underset{\nu^{K,u}_0\in B(\bar\xi^u,\gamma'), \nu^{K,v}_0 \in B(0,\epsilon)}{\sup} \, \Pb_{\nu^K_0}\left( R^K_{\gamma}\geq T^K_\epsilon \wedge S_1^K\wedge e^{KV} \right) \underset{K\to\infty}{\to}0,
 $$
 where $R^K_\gamma$, $T^K_\epsilon$ have been defined by~\eqref{def_RKgamma}, \eqref{def_TKepsilon}.\\
 Let assume that $\epsilon\leq \gamma$ and that $K$ is large enough such that $q_K\leq \gamma$. 
 Then on $\{t\leq R^K_\gamma \wedge T^K_\epsilon \wedge S_1^K\}$, the process $(\nu^{K,v}_t)_{t\geq 0}$ is stochastically bounded:
 \begin{equation*}
 \frac{1}{K}Z^{inf} \preceq \nu^{K,v} \preceq \frac{1}{K} Z^{sup},
 \end{equation*}
 $Z^{inf}$, $Z^{sup}$ are two branching diffusion processes starting with one individual at location $x_0$, their birth rates are respectively $b^v(x)(1-\gamma)$ and $b^v(x)$, and their death rates are $d^v(x)+c^{vu}\cdot\bar\xi^u+2\bar c \gamma$ and $d^v(x)+c^{vu}\cdot\bar\xi^u-\bar c \gamma$.\\
 Let us set 
 \begin{align*}
 &\Upsilon^{sup}_{K\epsilon}=\inf\{t\geq 0, \langle Z^{sup}_t,\one \rangle  \geq K\epsilon\},\\
 &\Upsilon^{sup}_{0}=\inf\{t\geq 0, \langle Z^{sup}_t,\one \rangle  = 0 \},
 \end{align*}
 and respectively $\Upsilon^{inf}_{K\epsilon}$, $\Upsilon^{inf}_0$ associated with $Z^{inf}$.\\
 Using same kind of computations as in Lemma 3 in \cite{champagnat_2006}, we deduce that
 \begin{equation}
 \label{eq_micmac}
  \begin{aligned}
  \Pb_{\nu^K_0} & \left(  \theta_0 \leq S_1^K\wedge \frac{\eta}{q_K K}, V_0=u, \W(\nu^K_{\theta_0},\bar \xi^u)<\gamma \right) \\
 & \geq \Pb_{\delta_{x_0}}\left(\Upsilon^{sup}_0 \leq \frac{\eta}{q_K K} \wedge \Upsilon^{sup}_{K\epsilon}\right)-\Pb_{\nu^K_0}\left(\frac{\eta }{q_K K} \geq S_1^K\right)-\Pb_{\nu^K_0}\left(\frac{\eta}{q_K K} \wedge S_1^K \wedge T^K_\epsilon \geq R^K_{\gamma}\right).
 \end{aligned}
 \end{equation}
Theorem~\ref{theo_exittimeintro} implies that the last term tends to $0$ under Assumption~\eqref{ass_KqK}. The second term tends also to $0$ as the number of individuals is stochastically bounded from above by a birth and death process with birth rate $\bar{b}$ and competition rate $\underbar c$, thus Lemma 2 in \cite{champagnat_2006} implies that for any $\delta>0$, there exists $\eta>0$ such that 
$
\limsup_{K\to+\infty}\, \Pb(S_1^K\leq \frac{\eta}{Kq_K})\leq \delta.
$\\ 
Thus the main difficulty is to evaluate the first term. Theorem~\ref{theo_survieinf} implies that 
\begin{equation}
 \label{eq_Tsup}
 \Pb_{\delta_{x_0}}\left(\Upsilon^{sup}_0 \leq \frac{\eta}{q_K K} \wedge \Upsilon^{sup}_{K\epsilon}\right)\underset{K\to +\infty}{\to} \Pb_{\delta_{x_0}}(\Upsilon_0^{sup} < +\infty)=1-\phi^{\gamma,vu}(x_0),
\end{equation}
where $\phi^{\gamma,vu}$ is the solution to the following elliptic equation on $\CX$ with Neumann boundary condition
\begin{equation*}
  m^v \Delta_x \phi^{\gamma,vu}(x)+[b^v(x)-d^v(x)-c^{vu}\cdot\bar\xi^u+\bar c \gamma]\phi^{\gamma,vu}(x)-b^v(x)\phi^{\gamma,vu}(x)^2=0.
\end{equation*}
Let us show that this solution is close to $\phi^{vu}$.
Theorem~\ref{theo_survieinf} implies that $\phi^{\gamma,vu}$ is positive if and only if 
\begin{equation}
\label{eq_cond}
H^v-c^{vu}\cdot\bar\xi^u+\bar c \gamma=H^v-\frac{H^u \kappa^{vu}}{\kappa^{uu}}+\bar c \gamma>0.
\end{equation}

\noindent
\textbf{First case}: $H^v\kappa^{uu}-H^u\kappa^{vu}<0$ (Point 1 in Assumption~\ref{ass_noncoexistence})\\
We can find $\gamma$ small enough such that \eqref{eq_cond} is not satisfied, thus $\phi^{\gamma,vu}\equiv \phi^{vu} \equiv 0$.
In addition with \eqref{eq_micmac} and \eqref{eq_Tsup}, we deduce that 
$$
\Pb_{\nu^K_0}  \left(  \theta_0 \leq S_1^K\wedge \frac{\eta}{q_K K}, V_0=u, \W(\nu^K_{\theta_0},\bar \xi^u)< \gamma \right) \underset{K\to+\infty}{\to} 1.
$$
Proposition~\ref{prop_dimo} is proved for this case.\\

\noindent
\textbf{Second case}: $H^v\kappa^{uu}-H^u\kappa^{vu}>0$ (Point 2 in Assumption~\ref{ass_noncoexistence})\\
Hence \eqref{eq_cond} is satisfied for all $\gamma>0$. Let $C$ be $\frac{\bar c \gamma}{\inf_{y\in \CX}b^v(y)\phi^{vu}(y)}$ and set 
$\mathcal{L}^{\gamma}(f)=m^v \Delta_x f+(b^v(x)-d^v(x)-c^{vu}\cdot\bar\xi^u+\bar c \gamma)f-b^v(x)f^2$. We have
\begin{align*}
 &\mathcal{L}^{\gamma}(\phi^{vu})=\bar c \gamma \phi^{vu}\geq 0,\\
 & \mathcal{L}^{\gamma}((1+C)\phi^{vu})=(1+C)\phi^{vu} [\bar c \gamma -C b^v \phi^{vu} ]\leq 0.
\end{align*}
As $\phi^{\gamma,vu}$ is the unique solution to $\mathcal{L}^{\gamma}(f)=0$, we deduce the following inequalities from a comparison theorem (see for example Theorem III.5 in \cite{roques_2013})~: for any $x\in \CX$,
\begin{equation}
\label{eq_inegalitephi}
 \left(1+C\right)\phi^{vu}(x) \geq \phi^{\gamma,vu}(x) \geq \phi^{vu}(x).
\end{equation}
We split the end of the proof into three steps regarding as the proof of Lemma 3 in \cite{champagnat_2006}.\\
Let us fix $\alpha>0$. \eqref{eq_micmac}, \eqref{eq_Tsup} and \eqref{eq_inegalitephi} imply that if $K$ is large enough,
\begin{equation*}
\Pb_{\nu^K_0}  \left( T^K_0\leq  \frac{\eta}{q_K K} \wedge S_1^K \wedge  R^K_{\gamma} \wedge T^K_\epsilon \right) \geq 1-\phi^{vu}(x_0)-\alpha.
\end{equation*} 
On the other hand, if $K$ is large enough :
\begin{align*}
& \Pb_{\nu^K_0}  \left( T^K_\epsilon\leq  \frac{\eta}{q_K K} \wedge S_1^K \wedge  R^K_{\gamma} \right) \\
 & \geq \Pb_{\delta_{x_0}}\left(T^{inf}_{K\epsilon} \leq \frac{\eta}{q_K K} \wedge T^{inf}_{0}\right)-\Pb_{\nu^K_0}\left(\frac{\eta }{q_K K} \geq S_1^K\right)-\Pb_{\nu^K_0}\left(\frac{\eta}{q_K K} \wedge S_1^K \wedge T^K_\epsilon \geq R^K_{\gamma}\right),\\
 &\geq \phi^{vu}(x_0)-\alpha.
\end{align*}
Thus, the $v$-process $\nu^{K,v}$ reaches a non-negligible size $K\epsilon$, before  $\frac{\eta}{q_K K} \wedge S_1^K \wedge  R^K_{\gamma}$, with a probability that tends to $\phi^{vu}(x_0)$.\\
Once the mutant population has reached a non-negligible size, we can compare the stochastic process and the deterministic limiting process. Under point 2 in Assumption~\ref{ass_noncoexistence}, there exist $T>0$ and $\gamma_2>0$ such that for any $\xi_0\in adh(B(\bar\xi^u,\gamma)\times (B(0,2\epsilon) \setminus B(0,\epsilon))$, for any $t\geq T$,
$$
\xi_t\in B(0,\gamma_2)\times B(\bar\xi^v,\gamma_2),
$$
where $\xi$ is the solution to \eqref{eq_theoapproxpop} with a dimorphic initial state $\xi_0\in M_F(\CX\times \{u,v\})$. This can be proved using similar arguments than those of the proof of Lemma~\ref{lemma_enterinball}.\\
Moreover, using Theorem~\ref{theo_majoration} and Proposition~\ref{prop_lowerboundI},
$$
\underset{\nu^K_0}{\sup} \, \Pb_{\nu^K_0}\left( \underset{t\in [0,T]}{\sup} \W(\nu^K_t,\xi_{t,\nu^K_0}) < \gamma_2 \right) \underset{K\to+\infty}{\to}0,
$$
with $\xi_{t,\nu^K_0}$ the solution to~\eqref{eq_theoapproxpop} with initial state $\nu^K_0$.
The two previous results and the Markov property imply that, if $K$ is large enough,
\begin{align*}
\Pb_{\nu^K_0} & \left( T^K_\epsilon\leq  \frac{\eta}{q_K K} \wedge S_1^K \wedge  R^K_{\gamma}, S_1^K \geq \frac{\eta }{q_K K} +T, \nu^{K,u}_{T^K_\epsilon+T}\in B(0,2\gamma_2), \nu^{K,v}_{T^K_\epsilon+T}\in B(\bar\xi^v,2\gamma_2) \right) \\
 &\geq \phi^{vu}(x_0)-2\alpha.
\end{align*}
Finally, we use the Markov property at time $T^K_\epsilon +T$ and we conclude as in Lemma 3 in \cite{champagnat_2006}. If $\gamma_2$ is sufficiently small, we prove that, with a probability that tends to $1$, after time $T^K_\epsilon+T$, the $u$-population process $\nu^{K,u}$ will become extinct before its size reaches the threshold $\sqrt{\gamma_2}$ and before the $v$-process $\nu^{K,v}$ moves away from a neighborhood of the equilibrium $\bar\xi^v$.\\
That concludes the proof of Proposition~\ref{prop_dimo} for the second case.
\end{proof}

Theorem~\ref{theo_main} is deduced from Propositions~\ref{prop_mono} and \ref{prop_dimo} in a similar way to Theorem 1 in \cite{champagnat_2006} by using the transition probabilities of the jump process $(\Lambda_t)_{t\geq 0}$. \\

\section{Numerics}
\label{sec_numerics}
To end this paper, let us illustrate Theorem~\ref{theo_main} with a numerical example. We consider here a set of parameters similar to the one in \cite{doebeli_dieckmann_2003} and \cite{leman_mirrahimi_meleard_2014}. The location space is $\CX=(0,1)$ and the trait space is $\CU=[0,1]$. For all trait $u$, we consider that the growth rate is maximal for the location $x=u$. For
instance, the location space state can represent a variation of resources, as seed size for some birds, and so two populations with two different traits are not best-adapted to same resources. For some bird species, a gradual variation of seed size can determine a gradual variation of beak size~\cite{grant_grant_2002}. Moreover, the maximum value of the growth rate function is the same for all traits but when $u$ increases, the birth rate function goes faster to $0$, as follows~:
$$
b(x,u)=\max\{4-160 \cdot u(x-u)^2,0\}, \quad d(x,u)=1.
$$
That is, the birds with a small trait value are more generalists than the ones with a large trait value : they are adapted to a larger set of seed size~\cite{futuyma_moreno_1988}. All individuals move with the same diffusion coefficient $m^u=0.003$. The competition kernel is a constant $c=10$ and the mutation kernel $k(x,u,w)$ is the probability density of a Gaussian random variable with mean $u$ and standard deviation $0.05$ conditioned on staying in $\CU$. The initial population is composed of $K$ individuals at location $x=0.5$ and with trait $u=0.6$. 

\begin{figure}[h!]
\begin{center}
 \begin{minipage}{0.32\textwidth}
 \centering
  \includegraphics[width=4.8cm]{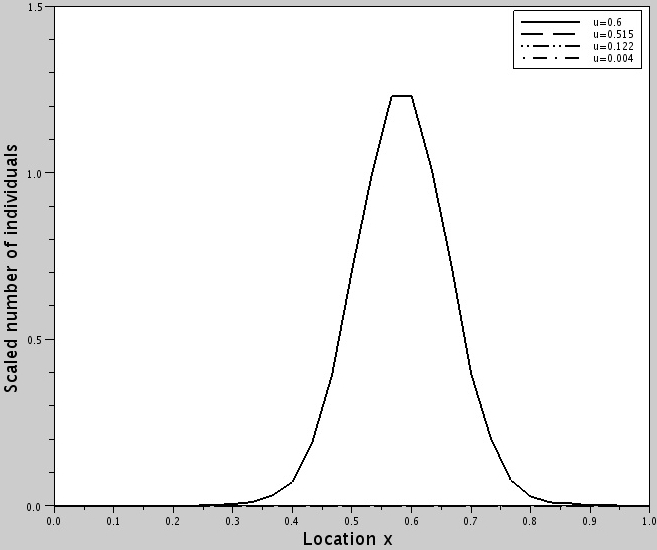}\\
  \small{ (a) t= 20}
 \end{minipage}
\begin{minipage}{0.32\textwidth}
\centering
  \includegraphics[width=4.8cm]{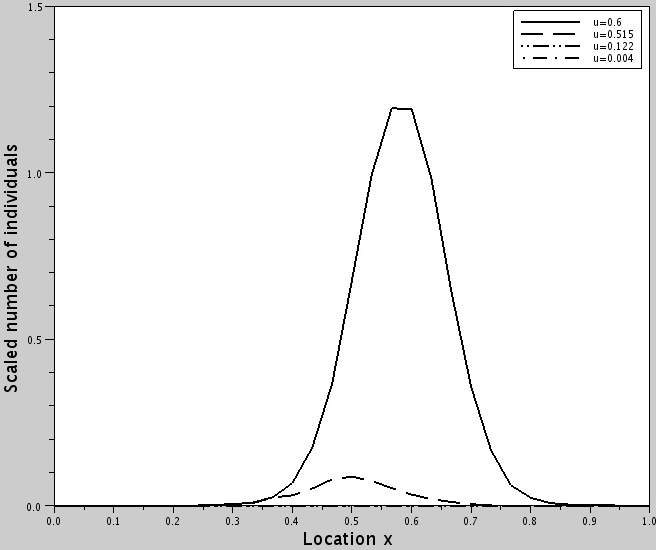}\\
  \small{ (b) t= 270}
 \end{minipage}
 \begin{minipage}{0.32\textwidth}
 \centering
  \includegraphics[width=4.8cm]{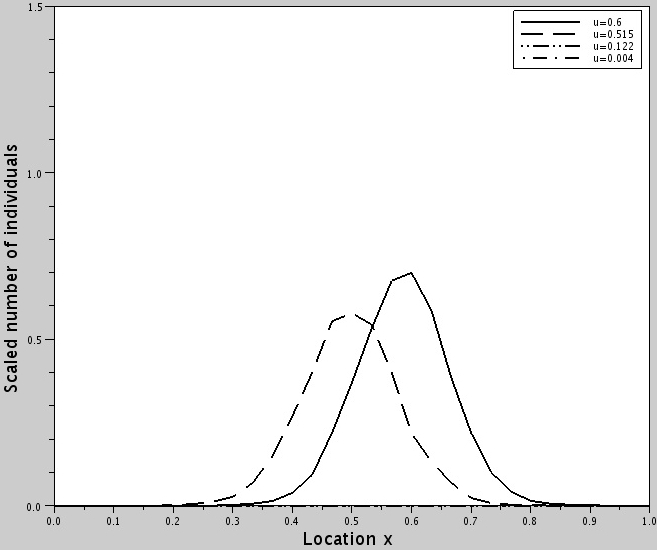}\\
  \small{ (c) t= 340}
 \end{minipage}
 
\begin{minipage}{0.32\textwidth}
\centering
  \includegraphics[width=4.8cm]{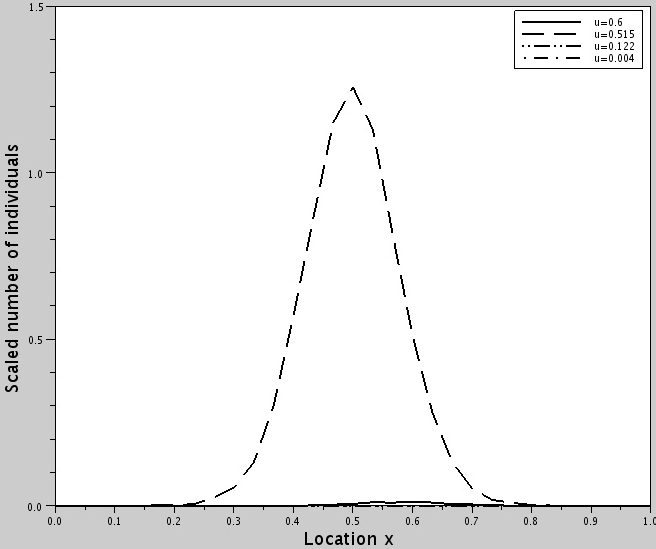}\\
  \small{ (d) t= 460}
 \end{minipage}
 \begin{minipage}{0.32\textwidth}
 \centering
  \includegraphics[width=4.8cm]{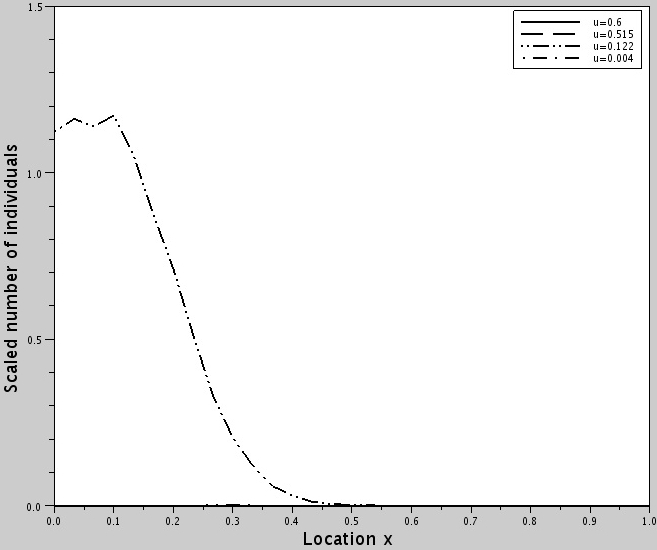}\\
  \small{ (e) t= 2 330}
 \end{minipage}
\begin{minipage}{0.32\textwidth}
\centering
  \includegraphics[width=4.8cm]{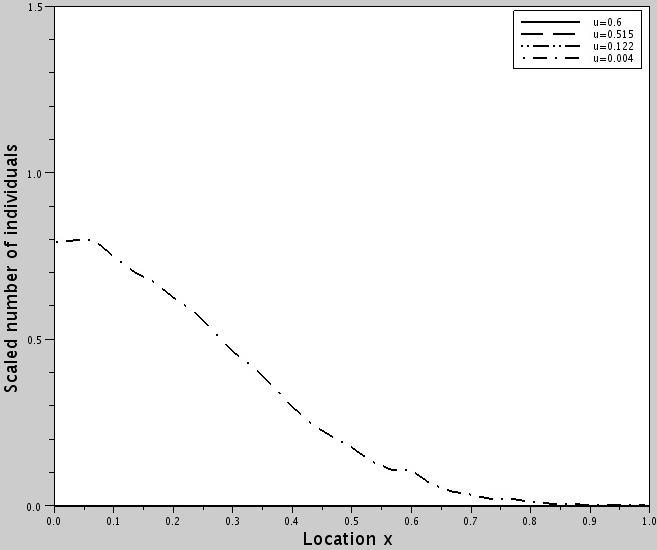}\\
  \small{ (f) t= 4 920}
 \end{minipage}
 \end{center}
 \caption{Simulations with $K=100 000$, $q_K=10^{-5}$.}
 \label{fig_simu}
\end{figure}

In Figure~\ref{fig_simu}, we observe the evolution of the trait associated with a change of spatial niches and spatial patterns over time. After a short time (a), the spatial distribution of the monomorphic population with trait $u=0.6$ stabilizes. Then, in Figures (b), (c) and (d), we observe a phenomenon of invasion and replacement~: some individuals with trait $u=0.515$ appear, invade and finally replace the previous population with trait $u=0.6$. Remark the change of spatial niche, see Figure (c). The locations of the population with trait $u=0.515$ are slightly smaller than the one of trait $u=0.6$. Three other phenomena of invasion and replacement with a displacement of the spatial niche are detected until the time $t=2330$ (Figure (e)). In a second phase, the population evolve to become more and more generalists (Figure (f)) : the length of the spatial niche is increasing at each event of invasion and replacement. \\ 
The simulations are computed using the algorithm presented in Parts 3 and 6 in Champagnat, M\'el\'eard~\cite{champagnat_meleard_2007}. It is an iterative construction, which gives an effective algorithm of the process. The diffusion motion is simulated using an Euler scheme for reflected diffusion process.\\

\noindent \textbf{Acknowledgments :} I would like to thank S. Méléard for
 her guidance during my work. I also thank C. Léonard, G. Raoul, C. Tran and A. Veber for their help. I acknowledge partial support by the ``Chaire Mod\'elisation Math\'ematique et Biodiversit\'e'' of VEOLIA - \'Ecole Polytechnique - MNHN - F.X.

\small
\bibliographystyle{plain}
\bibliography{biblio}

\begin{thebibliography}{10}

\bibitem{arnold_desvillettes_prevost_2012}
A.~Arnold, L.~Desvillettes, and C.~Pr{\'e}vost.
\newblock Existence of nontrivial steady states for populations structured with
  respect to space and a continuous trait.
\newblock {\em Commun. Pure Appl. Anal}, 11(1):83--96, 2012.

\bibitem{berestycki_hamel_roques_2005}
H.~Berestycki, F.~Hamel, and L.~Roques.
\newblock Analysis of periodically fragmented environment model: I-species
  persistence.
\newblock {\em J. Math. Biol.}, 51:75--113, 2005.

\bibitem{champagnat_2006}
N.~Champagnat.
\newblock A microscopic interpretation for adaptive dynamics trait substitution
  sequence models.
\newblock {\em Stochastic processes and their applications}, 116(8):1127--1160,
  2006.

\bibitem{champagnat_jabin_meleard_2014}
N.~Champagnat, P.E. Jabin, and S.~M{\'e}l{\'e}ard.
\newblock Adaptation in a stochastic multi-resources chemostat model.
\newblock {\em Journal de Math{\'e}matiques Pures et Appliqu{\'e}es},
  101(6):755--788, 2014.

\bibitem{champagnat_meleard_2007}
N.~Champagnat and S.~Méléard.
\newblock Invasion and adaptative evolution for individual-based spatially
  structured populations.
\newblock {\em J. Math. Biol.}, 55(2):147--188, 2007.

\bibitem{costa_hauzy_2014}
M.~Costa, C.~Hauzy, N.~Loeuille, and S.~M{\'e}l{\'e}ard.
\newblock Stochastic eco-evolutionary model of a prey-predator community.
\newblock {\em arXiv preprint arXiv:1407.3069}, 2014.

\bibitem{coville_2013}
J.~Coville.
\newblock Convergence to equilibrium for positive solutions of some
  mutation-selection model.
\newblock {\em arXiv preprint arXiv:1308.6471}, 2013.

\bibitem{dawson_gartner_1987}
D.~Dawson and J.~Gartner.
\newblock Large deviations from the mc-kean-vlasov limit for weakly interacting
  diffusions.
\newblock {\em Stochastics}, 20:247--308, 1987.

\bibitem{dembo_zeitouni_1998}
A.~Dembo and O.~Zeitouni.
\newblock {\em Large Deviations Techniques and Applications}.
\newblock Stochastic Modelling and Applied Probability, 1998.

\bibitem{desvillettes_ferriere_prevost_2004}
Desvillettes, Ferrière, and Prévost.
\newblock Infinite dimensional reaction-diffusion for population dynamics.
\newblock {\em preprint}, 4(3):529--605, 2004.

\bibitem{doebeli_dieckmann_2003}
M.~Doebeli and U.~Dieckmann.
\newblock Speciation along environmental gradients.
\newblock {\em Nature}, 421:259--263, 2003.

\bibitem{durrett_levin_1998}
R.~Durrett and S.~Levin.
\newblock Spatial aspects of interspecific competition.
\newblock {\em Theoretical population biology}, 53(1):30--43, 1998.

\bibitem{endler_1977}
J.A. Endler.
\newblock {\em Geographic variation, speciation, and clines}.
\newblock Number~10. Princeton University Press, 1977.

\bibitem{fontbona_2004}
J.~Fontbona.
\newblock Uniqueness for a weak nonlinear evolution equation and large
  deviations for diffusing particles with electrostatic repulsion.
\newblock {\em Stochastic processes and their applications}, 112(1):119--144,
  2004.

\bibitem{freidlin_wentzell_1984}
M.~Freidlin and A.~Wentzell.
\newblock {\em Random Perturbations}.
\newblock Springer, 1984.

\bibitem{friedman_1964}
A.~Friedman.
\newblock {\em Partial Differential Equations of Parabolic Type}.
\newblock Prentice-Hall, Englewood Cliffs, 1964.

\bibitem{futuyma_moreno_1988}
D.~Futuyma and G.~Moreno.
\newblock The evolution of ecological specialization.
\newblock {\em Ann. Rev. Ecol. Syst.}, 19:207--233, 1988.

\bibitem{graham_meleard_1997}
C.~Graham and S.~M{\'e}l{\'e}ard.
\newblock An upper bound of large deviations for a generalized star-shaped loss
  network.
\newblock {\em Markov Processes and Related Fields}, 3(2):199--224, 1997.

\bibitem{grant_grant_2002}
P.~Grant and B.~Grant.
\newblock Unpredictable evolution in a 30-year study of darwin's finches.
\newblock {\em Science}, 296:707--711, 2002.

\bibitem{kassen_2002}
R.~Kassen.
\newblock The experimental evolution of specialists, generalists, and the
  maintenance of diversity.
\newblock {\em J. Evol. Biol.}, 15:173--190, 2002.

\bibitem{leimar_doebeli_dieckmann_2008}
O.~Leimar, M.~Doebeli, and U.~Dieckmann.
\newblock Evolution of phenotypic clusters through competition and local
  adaptation along an environmental gradient.
\newblock {\em Evolution}, 62(4):807--822, 2008.

\bibitem{leman_mirrahimi_meleard_2014}
H.~Leman, S.~M\'el\'eard, and S.~Mirrahimi.
\newblock Influence of a spatial structure on the long time behavior of a
  competitive lotka-volterra type system.
\newblock {\em Discrete Contin. Dyn. Syst. Ser. B}, pages 469--493, 2015.

\bibitem{leonard_2001}
C.~L{\'e}onard.
\newblock Convex conjugates of integral functionals.
\newblock {\em Acta Math. Hungar.}, 93(4):253--280, 2001.

\bibitem{leonard_2001_bis}
C.~L{\'e}onard.
\newblock Minimizers of energy functionals.
\newblock {\em Acta Math. Hungar.}, 93(4):281--325, 2001.

\bibitem{leonard_1995}
C.~Léonard.
\newblock On large deviations for particle systems associated with spatially
  homogeneous boltzmann type equations.
\newblock {\em Prob.Th.Rel.Fields}, 101(1):1--44, 1995.

\bibitem{metz_geritz_MJV_1996}
JAJ. Metz, S.~Geritz, G.~Mesz{\'e}na, F.~Jacobs, and JS.~Van Heerwaarden.
\newblock Adaptive dynamics, a geometrical study of the consequences of nearly
  faithful reproduction.
\newblock {\em Stochastic and spatial structures of dynamical systems},
  45:183--231, 1996.

\bibitem{polechova_barton_2005}
J.~Polechov{\'a} and N.H. Barton.
\newblock Speciation through competition: a critical review.
\newblock {\em Evolution}, 59(6):1194--1210, 2005.

\bibitem{rao_ren_2002}
M.M. Rao and Z.D. Ren.
\newblock {\em Applications of Orlicz spaces}, volume 250.
\newblock CRC Press, 2002.

\bibitem{roques_2013}
L.~Roques.
\newblock {\em Modèles de réaction-diffusion pour l'écologie spatiale}.
\newblock Editions Quae, 2013.

\bibitem{tilman_kareiva_1997}
D.~Tilman and P.M. Kareiva.
\newblock {\em Spatial ecology: the role of space in population dynamics and
  interspecific interactions}, volume~30.
\newblock Princeton University Press, 1997.

\bibitem{tran_2008}
V.~Tran.
\newblock Large population limit and time behaviour of a stochastic particle
  model describing an age-structured population.
\newblock {\em ESAIM}, 12:345--386, 2008.

\bibitem{wang_yan_2013}
F.~Wang and L.~Yan.
\newblock Gradient estimate on convex domains and applications.
\newblock {\em Proceedings of the American Mathematical Society},
  141(3):1067--1081, 2013.

\end{thebibliography}

\end{document}